\documentclass{siamart}

\usepackage{graphicx} 
\usepackage{subfigure} 

\PassOptionsToPackage{numbers, compress}{natbib}
\usepackage{algorithm}
\usepackage{algorithmic}

\usepackage{hyperref}

\usepackage{url}
\usepackage[colorinlistoftodos,prependcaption,textsize=tiny]{todonotes}
\usepackage{amsmath,amssymb}

\usepackage{multirow}
\usepackage{amsfonts}
\usepackage{layout}
\usepackage{amsmath}
\usepackage{amssymb}
\usepackage{url}
\usepackage{color}
\usepackage{graphicx}
\usepackage{algorithmic}
\usepackage{algorithm}
\usepackage{verbatim}
\usepackage{xcolor}

\PassOptionsToPackage{normalem}{ulem}
\usepackage{ulem}

\usepackage{cases}

\providecolor{added}{rgb}{0,0,1}
\providecolor{deleted}{rgb}{1,0,0}

\definecolor{orange}{RGB}{255,127,0}

\newcommand\tagthis{\addtocounter{equation}{1}\tag{\theequation}}

\newcommand{\ve}[2]{\langle #1 ,  #2 \rangle}

\newcommand{\R}{\mathbb{R}}

\DeclareMathOperator{\Exp}{\mathbf{E}}           

\DeclareMathOperator{\dom}{dom}         

\newtheorem{assumption}{Assumption}

\newcommand*{\starnr}{\stepcounter{equation}\tag{\theequation}}

\usepackage{enumitem}

\newcommand{\G}{\mathcal{G}}
\newcommand{\vv}[2]{#1^{[#2]}}
\newcommand{\vf}[3]{#1^{[#2]}_{#3}}
\newcommand{\cR}{\mathcal{R}}
\newcommand{\vb}[2]{ {#1}^{[#2]}  }
\newcommand{\oneK}{1}

\newcommand{\alphav}{ {\boldsymbol \alpha}}

\newcommand{\bv}{ {\bf b}}

\newcommand{\uv}{ {\bf u}}
\newcommand{\vvv}{ {\bf v}}
\newcommand{\wv}{ {\bf w}}

\newcommand{\yv}{ {\bf y}}
\newcommand{\zv}{ {\bf z}}
\newcommand{\OA}{\mathcal{O}_{\hspace{-1pt}A}}
\newcommand{\OB}{\mathcal{O}_{\hspace{-1pt}B}}

\newcommand{\ConstantA}{C_1}
\newcommand{\ConstantB}{C_2}

\graphicspath{{./}{./exp/}}

\usepackage[utf8]{inputenc} 
\usepackage[T1]{fontenc}    
\usepackage{hyperref}       
\usepackage{url}            
\usepackage{booktabs}       
\usepackage{amsfonts}       
\usepackage{nicefrac}       
\usepackage{microtype}      
\allowdisplaybreaks

\title{An Accelerated Communication-Efficient Primal-Dual Optimization Framework for Structured Machine Learning\footnotemark[1]}

\author{Chenxin Ma\footnotemark[2]
   \and Martin Jaggi\footnotemark[3]
   \and Frank E.~Curtis\footnotemark[2]
   \and Nathan Srebro\footnotemark[4]
   \and Martin Tak\'a\v{c}\footnotemark[2]}
\date{\today}

\begin{document}

\maketitle

\renewcommand{\thefootnote}{\fnsymbol{footnote}}
\footnotetext[1]{This material is based upon work supported by the U.S. National Science Foundation, Division of Computing and Communications Foundations, under Award Number CCF-1618717.}
\footnotetext[2]{Department of Industrial and Systems Engineering, Lehigh University, Bethlehem, PA, USA; \email{chm514@lehigh.edu}, \email{frank.e.curtis@gmail.com}, \email{takac@lehigh.edu}}
\footnotetext[3]{School of Computer and Communication Sciences, \'Ecole Polytechnique F\'ed\'erale de Lausanne, Lausanne, Switzerland; \email{martin.jaggi@epfl.ch}}
\footnotetext[4]{Toyota Technological Institute at Chicago, Chicago, IL; \email{nati@ttic.edu}}
\renewcommand{\thefootnote}{\arabic{footnote}}

\begin{abstract}
Distributed optimization algorithms are essential for training machine learning models on very large-scale datasets.  However, they often suffer from communication bottlenecks.  Confronting this issue, a communication-efficient primal-dual coordinate ascent framework (\textsc{CoCoA}) and its improved variant \textsc{CoCoA+} have been proposed, achieving a convergence rate of $\mathcal{O}(1/t)$ for solving empirical risk minimization problems with Lipschitz continuous losses.  In this paper, an accelerated variant of \textsc{CoCoA+} is proposed and shown to possess a convergence rate of $\mathcal{O}(1/t^2)$ in terms of reducing suboptimality.  The analysis of this rate is also notable in that the convergence rate bounds involve constants that, except in extreme cases, are significantly reduced compared to those previously provided for \textsc{CoCoA+}.  The results of numerical experiments are provided to show that acceleration can lead to significant performance gains.
\end{abstract}

\section{Introduction}

The use of \emph{distributed} optimization has become essential for training machine learning models on very large datasets.  For example, distribution is a key ingredient when the dataset does not fit into the memory of a single machine, and rather must be stored in a distributed manner over many machines or computational agents, each having access only to their own (local) segment of the training data.

Efficiently training a machine learning model in such a network setting is challenging due to the cost of communicating information between machines, as compared to the relatively cheap cost of local computation on a single machine.  Therefore, designing efficient distributed optimization algorithms, which are able to balance the amount of local computation with the amount of necessary communication between machines, is of crucial importance.  This is especially the case since communication bottlenecks and the heterogeneity of large-scale computing systems have been increasing.



In this paper, we improve the existing \textsc{CoCoA+} framework for communication-efficient distributed optimization \cite{jaggi2014communication,ma2015adding,Smith:2016wp} by equipping it with \emph{acceleration}~\cite{nesterov2013introductory}.  We do not simply apply a generic acceleration scheme, e.g., using the Universal Catalyst~\cite{lin2015universal} (see \S\ref{sec.literature} for more discussion), but rather develop a technique specific for our setting.  We prove that our accelerated scheme possesses an improved convergence rate involving reduced constants compared to that of \textsc{CoCoA+}. Our theory holds for arbitrary local solvers used on the worker machines.  We also provide the results of extensive numerical experiments, where we demonstrate that the new accelerated scheme leads to significantly improved performance in real application settings.

\subsection{Preliminaries}

We present our algorithmic framework in the context of solving the following pair of optimization problems, which are dual to each other:
\begin{align}
    \label{eq:A}\tag{A}  
    \min_{\alphav \in \R^{n}} \ & \Big[ \ 
    \OA(\alphav) := f(A\alphav )
    \, +\, \textstyle\sum_{i\in\mathcal{I}} g_i(\alpha_i) 
     \Big]\quad \text{and}\,
\\
    \label{eq:B}\tag{B} 
    \min_{\wv \in \R^{d}} \ & \Big[ \ 
    \OB(\wv) := f^*(\wv )
    \, +\, \textstyle\sum_{i\in\mathcal{I}} g_i^*(-A_i^\top\wv)  
      \Big].
\end{align}
Here, $h^*$ denotes the \textit{convex conjugate} of a function~$h$.  The two problems are defined by  the functions $f : \R^d \rightarrow \R$ and $g_i : \R \rightarrow \R$ for all $i \in \mathcal{I}$ and a data matrix $A\in\R^{d\times n}$, about which we make the following assumption.

\begin{assumption}\label{assumption.main}
  The functions $f$ and $g_i$ for all $i \in \mathcal{I}$ are closed and convex.  In addition, the gradient function $\nabla f$ is Lipschitz continuous with constant $\lambda$ and the function $g_i^*$ for each $i \in \mathcal{I}$ is Lipschitz continuous with constant $L$.  Finally, without loss of generality, the columns of the data matrix satisfy $\|A_i\| \leq 1$.
\end{assumption}


Observe that problem formulation~\eqref{eq:A} includes many important problems with nonsmooth regularizers, such as Lasso with $\ell_1$-norm regularization, i.e.,
\begin{align}\label{eq:L1}
    \min_{\alphav \in \R^{n}} \ & 
    \tfrac{1}{2}  \left\|A \alphav-\bv\right\|_2^2  
    + \lambda_1 \|\alphav\|_1.
\end{align}
More generally, many other smooth data-fitting functions can be used in place of $f$ in our context of problem~\eqref{eq:A}.  Problem formulation \eqref{eq:B} includes many regularized empirical loss minimization problems with strongly convex regularizers of the form\vspace{-2mm}
\begin{align}\label{eq:ERM}
    \min_{\wv \in \R^{d}} \ & \tfrac1n \sum_{i=1}^n \ell_i( A_i^\top \wv) 
 + \tfrac{\lambda}{2} \|\wv\|^2,
\end{align}
such as support vector machines (SVMs) or logistic regression.  Here, the $i$-th column $A_i \in \R^d$ of $A$ represents the $i$-th data example and $\lambda>0$ is a regularization parameter.
\subsection{Motivation and Literature Review}\label{sec.literature}

The increased size of optimization problems of interest in machine learning, as well as the availability of parallel and distributed architectures, has led to various directions of research on the design of parallel algorithms for either solving \eqref{eq:L1} directly (e.g., see \cite{bradley2011parallel,recht2011hogwild,duchi2011adaptive}) or indirectly by solving its dual \eqref{eq:ERM} (e.g., see
\cite{takavc2013mini,liu2015asynchronous,takavc2015distributed}).  For example, some previous work has focused on shared memory systems in which the amount of parallelism possible is dependent upon the number of processing units (CPUs) on a single node, which is usually in the range of 16--64 cores per node in a contemporary high performance cluster.  To avoid locking and, hence, improve speed, some \emph{asynchronous} algorithms have been proposed and analyzed 
\cite{liu2015asynchronous,recht2011hogwild}.  However, despite the fact that such methods are efficient with shared memory, na\"ive extensions of these ideas to distributed environments can be terribly inefficient.  This is due to the fact that such na\"ive extensions require constant communication between nodes, causing them to suffer large overhead.

This type of observation has led many to conclude that, in a distributed computing environment, one must focus not only on the number of data accesses, but also on the number of communication steps, which is usually tied to the number of iterations of the overall optimization algorithm~\cite{ma2015distributed}.  For this reason, researchers have proposed the use of batch (i.e., full gradient) methods, potentially using (partial or approximate) second-order information as in Newton's method \cite{shamir2013communication,zhang2015communication} or a quasi-Newton method such as L-BFGS \cite{zhu1997algorithm}.  The benefit of such methods is that they typically require fewer iterations to achieve high accuracy of the solution.  On the other hand, a Newton-type method requires that a linear system be solved (approximately) in every iteration.  A popular approach for this procedure is the linear conjugate gradient (CG) algorithm, such as proposed in \textsc{Disco}~\cite{zhang2015communication}.  Unfortunately, however, in many applications one might need a substantial number of CG steps in each Newton iteration.  Since each CG step requires one pass over the data and one round of communication, this can also lead to substantial communication overhead.
 
To overcome these issues, the algorithms \textsc{CoCoA} \cite{jaggi2014communication}, DisDCA \cite{yang2013analysis} and \textsc{CoCoA+} \cite{ma2015adding,Smith:2016wp} have been proposed to efficiently balance computation and communication in distributed optimization environments.  The main idea of these methods, with the data partitioned across any number of nodes in a cluster, is to define meaningful auxiliary subproblems to be solved in each node using only locally stored data.  These subproblems involve some inherited (partial) second-order information, which aids in yielding overall fast convergence for solving the original problems~\eqref{eq:A} and~\eqref{eq:B}.

As is the case for many other optimization methods for machine learning, the existing \textsc{CoCoA+} framework only yields a sublinear convergence rate of $\mathcal{O}(1/t)$ for the general convex case.  That said, we are motivated by the fact that several single machine solvers can be improved by incorporating Nesterov acceleration~\cite{nesterov2013introductory,beck2009fast}, leading to an improved rate of $\mathcal{O}(1/t^2)$; e.g., this was recently successfully done in \cite{shalev2013accelerated}. In this paper, we are able to provably accelerate the distributed \textsc{CoCoA+} framework, achieving a $\mathcal{O}(1/t^2)$ rate for reducing suboptimality.

It is worthwhile to mention that acceleration could be achieved using the Universal Catalyst proposed in \cite{lin2015universal}.  However, for our purposes, this approach is less appealing since it does not allow the local solver to produce randomized solutions which might only have sufficient quality in expectation.  Perhaps the accelerated algorithms most relevant to the present work are those in \cite{nesterov2013gradient,fercoq2015accelerated}. 
The subproblems in these papers are assumed to be strongly convex.  In our case, this is not general enough; instead, we manage to exploit the general structure of the objective $\OA$.

\subsection{Contributions}

Overall, the contributions represented by the work in this paper can be summarized as follows.  Here, for any $\alphav \in \R^{n}$ in the context of \eqref{eq:A}, we define a corresponding dual solution $\wv\in \R^{d}$ for problem~\eqref{eq:B} as
\begin{equation}
\label{eq:dualPdualrelation}
\wv = \wv(\alphav) := \nabla f( A\alphav ) \, .
\end{equation}

\begin{itemize}
[leftmargin=0.5cm,topsep=0pt,itemsep=1ex,partopsep=0ex,parsep=0ex]
  \item We propose and analyze an accelerated communication-efficient block descent algorithm for solving problems~\eqref{eq:A} and~\eqref{eq:B}.  Our analysis shows that our framework possesses a convergence rate of $\mathcal{O}(1/t^2)$ for suboptimality, i.e., with $\alphav_t$ representing the $t$-th algorithm iterate, we prove the accelerated reduction of $\OA(\alphav_t) - \OA(\alphav_\star)$ over time $t$, where $\alphav_\star$ represents an optimal solution of \eqref{eq:A}.  Given the recent work in~\cite{dunner2016primal}, our results for reducing suboptimality can then also be cast in terms of an accelerated rate of the duality gap $G(\alphav) :=  \OA(\alphav) + \OB(\wv(\alphav))$, as a practically important accuracy certificate.
  \item The convergence analyses of \textsc{CoCoA} and \textsc{CoCoA+} yielded bounds involving quantities which depend on 
  the  Lipschitz constants for the local subproblems. In this sense, the results might be no better than similar bounds for inexact block proximal gradient descent.  By contrast, in this paper, we exploit the structure of the dual objective $\OB$ to obtain complexity bounds that do not depend on Lipschitz constants of the auxiliary subproblems.  Instead, our bounds merely depend on quantities related to local curvature corresponding to the subproblems.
  \item We extend our accelerated framework to also cover general non-strongly convex regularizers.  This is important, e.g., in the context of $\ell_1$-norm regularized Lasso, sparse logistic regression, and elastic net regularized problems.
  \item We have performed numerous numerical experiments to demonstrate that acceleration can lead to significant performance improvements.  In particular, the gains are especially large for small values of the regularization parameter $\lambda$, which is important in very-large-scale settings where one often desires the regularization to be inversely proportional to $n$.  The C++ code for our implementation of our framework is available on github: \url{https://github.com/schemmy/CoCoA-Experiments}.
\end{itemize}

\subsection{Organization}

The remainder of the paper is organized as follows. We start by introducing our new accelerated \textsc{CoCoA+} (\textsc{AccCoCoA+}) algorithm, including the design of its subproblems and strategies for solving them.  We then describe our main complexity result for the algorithm, showing its improvement over that for \textsc{CoCoA+}.  Finally, we comment on the results of our  numerical experiments.

\section{Accelerated \textsc{CoCoA+}}
\label{sec:acocoa}
In this section, we introduce our proposed accelerated \textsc{CoCoA+} (\textsc{AccCoCoA+}) algorithm.  We begin by defining notation related to the manner in which data is distributed across various machines, then define quantities related to the local subproblems to be solved in each iteration of the algorithm.  Of central importance for these subproblems are a carefully defined regularization scheme and a loose assumption on the accuracy to which each subproblem must be solved in each step of \textsc{AccCoCoA+}.

\subsection{Data Partitioning}
 
Suppose that the $n$ columns of the dataset are split across $K$ machines (nodes).  Let the index set of columns stored on node $k \in [K] := \{1,\dots,k\}$ be denoted as~$\mathcal{P}_k$.  We assume that the dataset is partitioned in a disjoint manner such that $\mathcal{P}_i \cap \mathcal{P}_j = \emptyset$ for any $i\neq j$ while $\cup_{k=1}^K \mathcal{P}_k = [n]$.  For notational convenience, we split the variable vector $\alphav \in \R^n$ into the set of $K$ vectors $\{\vv{\alphav}{k}\}_{k=1}^K$ by employing, for each $k \in [K]$, the masking operator
\begin{equation}\label{eq:partition}
{(\vv{\alphav}{k})}_i
 :=
 \begin{cases}
 \alpha_i&\mbox{if}\ i\in \mathcal{P}_k,\\
 0&\mbox{otherwise.}
\end{cases}
\end{equation}
As a consequence of separability of $g$ in \eqref{eq:A}, we write
\begin{equation}\label{asfwfrwafwafas} 
  \psi_k (\vv{\alphav}{k}) := \sum_{i\in\mathcal P_k} g_i( \alpha_i ),
\end{equation}
so that $g(\alphav) = \sum_{k=1}^K \psi_k(\vv{\alphav}{k})$.

\subsection{Subproblem}
 
Iteration $t$ of \textsc{AccCoCoA+} involves the auxiliary vectors $(\yv_t, \zv_t) \in \R^n \times \R^n$, which one may split into $\{\vv{\yv_t}{k}\}_{k=1}^K$ and $\{\vv{\zv_t}{k}\}_{k=1}^K$, respectively, in the same manner as $\alphav \in \R^n$ in~\eqref{eq:partition}.  Fundamentally, the goal in iteration $t$ on each node $k \in [K]$ is to (approximately) solve
\begin{align}\label{eq:subpdef3}
  \min_{\vv{\zv_{t+1}}{k} \in\R^{n}} \G_k(\vv{\zv_{t+1}}{k};\yv_t,\zv_t), 
\end{align}
where, for scalars $\sigma' \geq 0$ and $\theta_t \geq 0$ (see below), the local objective function is
\begin{align}\label{eq:subpdef4}
 \G_k(  \vv{\zv_{t+1}}{k}  ; \yv_t, \zv_t)  :=&\ 
  \psi_k(\vv{\zv_{t+1}}{k}) +\frac1Kf(A \yv_t) \notag\\
  & +  \nabla f(A \yv_t)^\top A (\vv{\zv_{t+1}}{k} - \vf{\yv}{k}{t})
   + \frac{\lambda \theta_t\sigma'}{2}\left\| A (\vv{\zv_{t+1}}{k}- \vf{\zv}{k}{t} )\right\|^2.
\end{align}

At first glance, it is not obvious that this subproblem can be solved only using local data on node~$k$ due to the presence of the term $\nabla f(A \yv_t)^\top A (\vv{\zv_{t+1}}{k} - \vf{\yv}{k}{t})$, which is dependent on the entire dataset.
However, by simply making the single shared vector $\wv_t := \nabla f(A \yv_t)$ available on each node, the local subproblem~\eqref{eq:subpdef3} only requires knowledge of the pair $(\vv{\yv_t}{k},\vv{\zv_t}{k})$ and the local part of $A$, and not the full vectors $(\yv_t,\zv_t)$.  Therefore, the storage of the variable vectors $\yv$ and $\zv$ can also be distributed.

The last term in \eqref{eq:subpdef4} represents a regularization term in which the parameter $\sigma' \geq 0$ plays a critical role.  In short, it can be interpreted as a measure for the cross-dependency of the partitioning of the data.  For our analysis, this term must be chosen, for some $\gamma \in [\tfrac1K,1]$, to satisfy
\begin{equation}\label{eq:sigmaineq}
  \sigma' \geq \sigma'_{min} := \gamma \max_{\alphav\in\\R^n}\frac{\|A \alphav\|^2}{\sum_{k=1}^{K} \|A \alphav^{[k]}\|^2}.
\end{equation}
It is easy to show that $\sigma'_{min}/\gamma = \max_{\alphav\in\\R^n}\frac{\|A \alphav\|^2}{\sum_{k=1}^{K} \|A \alphav^{[k]}\|^2} \in [1,K]$.  If it is equal to 1, then any pair of samples from different elements of the partition must be 
orthogonal to each other.  In such cases, the function $\OA$ is block-separable.  On the other hand, if $\sigma'_{min}/\gamma$ is close to $K$, then the data across the partition are strongly correlated.  Note that the choice $\sigma' = \gamma K$ is ``safe'' in the sense that \eqref{eq:sigmaineq} holds; see \cite[Lemma~4]{ma2015adding}. 

\subsection{Approximate Subproblem Solutions}

A strength of our framework is that each subproblem \eqref{eq:subpdef3} need not be solved exactly.  This is critical since, in the extreme, solving the subproblems exactly can be as difficult as solving the original problem.  In \textsc{AccCoCoA+}, we make the assumption that, in iteration $t$, the solver employed to solve the subproblem on node~$k$ yields an approximate solution with some additive error $\epsilon_t \geq 0$.  To be precise, we make the following assumption.
\begin{assumption}[$\epsilon_t$-approximate solutions]\label{asm:epsilon} 
  There exists a sequence $\{\epsilon_t\}_{t=0}^{\infty}  \geq 0$ such that, for each $t \in \{0,1,2,\dots\}$ and $k \in [K]$, the local solver employed on node $k$ in iteration $t$ produces
a (possibly random) $\vv{\zv_{t+1}}{k}$ satisfying
\begin{equation}\label{eq:subprostate}
  \Exp_{t+1}[\G_k(\vv{\zv_{t+1}}{k} ; \yv_t, \zv_t) | t ] \leq \G_k(\vv{{\zv_{t+1}^\star}}{k}; \yv_t, \zv_t) + \epsilon_t,
\end{equation}
where $\vv{{\zv_{t+1}^\star}}{k} :=  \arg\min_{\vv{\zv_{t+1}}{k} \in \R^n} \G_k(\vv{\zv_{t+1}}{k}; \yv_t, \zv_t)$ and $\Exp_{t+1}[\cdot | t]$ indicates conditional expectation given the algorithm history up to time $t$.
\end{assumption}

\subsection{Algorithm}

\textsc{AccCoCoA+} is stated as Algorithm~\ref{alg:algo1}.   Given an initial iterate vector $\alphav_0$ in the effective domain of $g$ and the scalar $\theta_0 = 1$, each iteration involves a series of steps, only one of which involves communication between nodes.  First, the auxiliary vectors $\{\vv{\yv_t}{k}\}_{k=1}^K$ are set on each node, each representing a convex combination of the variables $\vv{\alphav_t}{k}$ and $\vv{\zv_t}{k}$.  Then, for setting up the local objective~\eqref{eq:subpdef4} for each subproblem~\eqref{eq:subpdef3}, the combined vector $\yv_t$ is used to compute $\wv_t$, which must be communicated to all nodes.  After this point in iteration $t$, all remaining steps involve local computation on each node: each subproblem is solved approximately to compute $\{\vv{\zv_{t+1}}{k}\}_{k=1}^K$, after which the elements of $\{\vv{\alphav_{t+1}}{k}\}_{k=1}^K$ are set.  Acceleration of the algorithm is due to the careful update for the sequence $\{\theta_t\}$, which, since it only involves a prescribed formula for a scalar quantity, can be performed identically on each node.  Observe that the update sequence ensures that $\theta_t \sim \mathcal{O}(1/t)$.

\begin{algorithm}
  \caption{Accelerated \textsc{CoCoA+} (\textsc{AccCoCoA+})}
  \label{alg:algo1}
\small 
  \begin{algorithmic}[1]
    \STATE choose $\alphav_0 \in \dom(g) \subseteq \mathbf R^n$; set $\zv_0:=\alphav_{0}$, $\gamma \in [\tfrac1K,1]$ and $\theta_0:=1$
    \FOR {$t \in \{0,1,2,\dots\}$} 
      \STATE \textbf{for} $k \in [K]$ in parallel, set
      \begin{equation}\label{afwgrvwfeagwefrw}
        \vf{\yv}{k}{t} := (1-\gamma\theta_t) \vf{\alphav}{k}{t} + \gamma\theta_t \vf{\zv}{k}{t} \vspace{-3mm}
      \end{equation}     
      \STATE set $\wv_t := \nabla f(A \yv_t)$ and communicate to all nodes
      \STATE \textbf{for} $k \in [K]$ in parallel, compute an $\epsilon_t$-approximate solution $\vv{\zv_{t+1}}{k}$ of subproblem \eqref{eq:subpdef3} 
      \STATE \textbf{for} $k \in [K]$ in parallel, set 
      \begin{equation}\label{afasfasfdsafa}
        \vf{\alphav}{k}{t+1}  := \vf{\yv}{k}{t} + \gamma\theta_t (\vf{\zv}{k}{t+1} - \vf{\zv}{k}{t}) \vspace{-3mm}
      \end{equation}
      \STATE set $\theta_{t+1} := \frac{\sqrt{\gamma^2\theta_t^4 + 4 \theta_t^2} - \gamma\theta_t^2}{2}$ (on each node) \label{eq.thetat1}
    \ENDFOR
  \end{algorithmic}
\end{algorithm}

\section{Convergence Analysis}

In this section, we study the convergence properties of the proposed \textsc{AccCoCoA+} algorithm.  First, we prove general complexity results, then, respectively in Sections \ref{sec:exact} and \ref{sec:inexact}, we provide interpretations of this main theorem for cases when the subproblems are solved exactly or inexactly.

First, we prove the following lemma related to the sequence $\{\theta_t\}_{t\geq0}$.  The result is similar to that given as Lemma 1 in \cite{fercoq2015accelerated}.

\begin{lemma}\label{lemma1} 
  The sequence $\{\theta_t\}_{t\geq 0}$ is positive, monotonically decreasing, and has
  \begin{equation}\label{eq:34}
    \frac{1-\gamma\theta_{t+1}}{\theta^2_{t+1}} = \frac{1}{\theta^2_t}
  \end{equation}
  and
  \begin{equation}\label{afsdfasfadsfdsa}
    \theta_t \leq \frac{2 }{t\gamma +2 } \leq 1
  \end{equation}
  for all $t \geq 0$.
\end{lemma}
\begin{proof}
  For each $t \geq 0$, the value $\theta_{t+1}$ can be seen from Step~\ref{eq.thetat1} of Algorithm~\ref{alg:algo1} to be the positive root of the quadratic equation
  \begin{equation}\label{eq.quadratic}
    \theta^2 + (\gamma\theta_t^2) \theta - \theta_t^2 = 0.
  \end{equation}
  Since $\theta_0 = 1$, it follows from \eqref{eq.quadratic}, the fact that a strongly convex quadratic univariate function with a negative vertical intercept has a positive real root, and a simple inductive argument that $\theta_t > 0$ for all $t \geq 0$, as desired.  Next, plugging in $\theta_{t+1}$ for $\theta$ in \eqref{eq.quadratic} and rearranging, we obtain~\eqref{eq:34}.  This can again be rearranged to yield
  \begin{equation}\label{eq:34_rewrite}
    \frac{1}{\theta_{t+1}^2} = \frac{1}{\theta_t^2} + \frac{\gamma}{\theta_{t+1}},
  \end{equation}
  from which it follows that
  $$
  \frac{1}{\theta_t^2} =1+ \sum_{i=1}^{t} \frac{\gamma}{\theta_i}\ \ \text{for all $t \geq 0$}.
  $$
  This shows that $\{\theta_t\}_{t\geq0}$ is monotonically decreasing, as desired.  We now use mathematical induction to show \eqref{afsdfasfadsfdsa}.  First, \eqref{afsdfasfadsfdsa} clearly holds for $t=0$, for which we have
$\theta_0 = \frac{2 }{0+2  } = 1$.  Assuming it holds up to $t$, we have from \eqref{eq:34_rewrite} that
  \begin{equation}\label{eq:needtouse918}
    \frac{1}{\theta_{t+1}^2} -\frac{\gamma}{\theta_{t+1}} =
\frac1{\theta_t^2}
\geq \frac{(\gamma t+2)^2}{4}.
  \end{equation}
  Now observe that the quadratic equation in the variable $1/\theta$ given by
  \begin{align*}
    \frac{1}{\theta^2} - \frac{\gamma}{\theta} -  \frac{(\gamma t+2)^2}{4} = 0\ \ \text{has roots}\ \ \frac{1}{\theta} = \frac{\gamma \pm\sqrt{\gamma^2 + (\gamma t+2)^2 } }{2}.
\end{align*}
  This shows that, by~\eqref{eq:needtouse918} and since $\theta_{t+1} > 0$, we have
\begin{align*}
  \frac{1}{\theta_{t+1}}
  \geq
      \frac{  \gamma + \sqrt{  \gamma^2 + (\gamma t+2)^2   }  }{2}
  \geq
      \frac{  \gamma + \sqrt{    (\gamma t  + 2)^2   }  }{2}
  =
      \frac{     \gamma (t+1)  + 2     }{2}.
\end{align*} 
Therefore, we conclude $\theta_{t+1}\leq \frac{2}{(t+1)\gamma   +2}\leq 1$, which concludes the proof.

\end{proof}

Next, we prove the following lemma, which is a modification of Lemma~2 in \cite{fercoq2015accelerated}.
 
\begin{lemma}\label{lemma2}
Let $\{\alphav_t,\zv_t\}_{t\geq 0}$ be generated by Algorithm~\ref{alg:algo1}. Then, for all $t\geq0$,
\begin{equation}\label{eq:35}
\alphav_t =   \sum_{l=0}^t {\rho^l_t \zv_l},
\end{equation}
where the coefficients $\{\rho^0_t, \rho^1_t, ..., \rho^t_t\}$ are nonnegative and sum to 1; i.e., $\alphav_t$ is a convex combination of the vectors $\{\zv_0, \zv_1,..., \zv_t\}$.  More precisely, the coefficients are defined recursively in $t$ as $\rho_0^0 = 1, 
\rho_1^0 = 1-\gamma \theta_0, \rho_1^1 = \gamma \theta_0$ and, for all $t\geq 1$,
\begin{align} \label{eq:36}
\rho^l_{t+1} = 
\begin{cases}
(1-\gamma\theta_t)\rho_t^l& \text{for}\ \ l \in \{0,...,t\},\\
 \gamma \theta_t  & \text{for}\ \ l=t+1.
\end{cases}
\end{align}
\end{lemma}
\begin{proof}
We proceed by induction. First, notice that $\alphav_0 = \zv_0 = \rho _0^0\zv_0$ where $\rho_0^0 := 1$.  By \eqref{afwgrvwfeagwefrw}, this implies that $\yv_0=\zv_0$, which together with 
$\theta_0=\oneK$ gives (see \eqref{afasfasfdsafa})
$$\alphav_1 
   = \yv_0+ \gamma\theta_0(\zv_1-\zv_0) 
   = \zv_0+ \gamma\theta_0(\zv_1-\zv_0) 
   =  (1-\gamma \theta_0) \zv_0 + \gamma \theta_0 \zv_1,$$
   which proves \eqref{eq:35} for $t = 1$.  Assuming now that $\eqref{eq:35}$ holds for some $t\geq 1$, we obtain
\begin{align} 
\alphav_{t+1} 
&\overset{\eqref{afasfasfdsafa}}{=}
   \yv_t + \gamma\theta_t(\zv_{t+1}-\zv_t)\notag\\
&\overset{\eqref{afwgrvwfeagwefrw}}{=}
   (1-\gamma\theta_t)  \alphav_t + \gamma\theta_t \zv_t + \gamma\theta_t(\zv_{t+1}-\zv_t)\notag\\
&\overset{\eqref{eq:35}}{=}
   (1-\gamma\theta_t) \sum_{l=0}^t {\rho^l_t \zv_l} + \gamma\theta_t \zv_t + \gamma\theta_t(\zv_{t+1}-\zv_t)\notag\\
&\overset{\eqref{eq:35}}{=}
    \sum_{l=0}^{t} 
     \underbrace{(1-\gamma\theta_t)  \rho^l_t}_{\rho_{t+1}^l} \zv_l  
     + \underbrace{\gamma\theta_t}_{\rho^{t+1}_{t+1}}  \zv_{t+1}.
     \label{eq.convcomb}
\end{align}
From \eqref{afsdfasfadsfdsa} and since $\gamma \in [\frac{1}{K},1]$, it follows that $\rho^l_{t+1}\geq 0$ for all $l \in \{0,\dots,t+1\}$. It remains to show that the constants sum to 1.  This is true since, for all $t \geq 1$, the weights in \eqref{eq.convcomb} (for $\alphav_{t+1}$) are obtained by taking the corresponding weights for $\alphav_t$, multiplying them by $\gamma\theta_t \in (0,1]$, then including the weight for $\zv_{t+1}$ as $\gamma\theta_t$.
\end{proof}

Our next result relates to the optimal solution of an instance of subproblem \eqref{eq:subpdef3}.

\begin{lemma}\label{lemma3}
Let 
$\vv{{\uv^\star}}{k}:=\arg\min_{\vv{\uv}{k} } 
  \G_k(  \vb{\uv}{k}  ; \yv , \zv)   
$.
Then, for all $\vv{\uv}{k}$, 
\begin{equation}\label{asdfsafdsafsfsa}
\G_k(  \vb{\uv}{k}  ; \yv , \zv)   \geq  \G_k( \vv{{\uv^\star}}{k} ; \yv , \zv)  + \frac{\lambda \theta_t\sigma'}{2} \| A(
 \vb{\uv}{k} - \vv{{\uv^\star}}{k} ) \|^2.
\end{equation}

\end{lemma}
\begin{proof}
  Using convexity of $f$ and $g_i$ for all $i \in \mathcal{I}$ (and, therefore, convexity of $\psi_k$ defined in \eqref{asfwfrwafwafas}), one finds that for $\vv{\vvv}{k} \in \partial \psi_k(\vb{{\uv^\star}}{k})$ one has
\begin{align*}
  &\ \G_k(  \vb{\uv}{k}  ; \yv , \zv)  
\\
     \overset{\eqref{eq:subpdef4}}{=}&\   
     \psi_k(\vb{\uv}{k}) +\frac1K f(A \yv)
   + \nabla f(A \yv)^\top A(\vb{\uv}{k} - \vv{\yv}{k})
   + \frac{\lambda \theta_t\sigma'}{2}
     \left\| A (\vv{\uv}{k}- \vv{\zv}{k} )\right\|^2
\\
\geq&\ 
     \psi_k(\vb{{\uv^\star}}{k}) 
   + \ve{\vv{\vvv}{k}}{ \vv{\uv}{k} - \vb{{\uv^\star}}{k}  }  
\\&\qquad
   + \frac1K f(A \yv)
   + \nabla f(A \yv)^\top A (\vb{\uv}{k}-\vb{{\uv^\star}}{k}+ \vb{{\uv^\star}}{k}- \vv{\yv}{k}  ) 
\\&\qquad 
   + \frac{\lambda \theta_t\sigma'}{2}
        \left\| A (\vv{\uv}{k} -\vb{{\uv^\star}}{k}+ \vb{{\uv^\star}}{k}- \vv{\zv}{k} )\right\|^2
\\ =&\ 
    \psi_k(\vb{{\uv^\star}}{k}) 
   + \ve{\vv{\vvv}{k} }{ \vv{\uv}{k} - \vb{{\uv^\star}}{k}  }  
\\&\qquad 
   + \frac1K f(A\yv)
   + \nabla f(A \yv)^\top A (\vb{\uv}{k}-\vb{{\uv^\star}}{k}  ) 
   + \nabla f(A \yv)^\top A ( \vb{{\uv^\star}}{k}- \vv{\yv}{k}  ) 
\\ &\qquad 
   + \frac{\lambda \theta_t\sigma'}{2}
        \left\| A (\vv{\uv}{k} -\vb{{\uv^\star}}{k})\right\|^2
   + \frac{\lambda \theta_t\sigma'}{2}
        \left\| A (\vb{{\uv^\star}}{k}- \vv{\zv}{k} )\right\|^2
\\&\qquad 
   + \lambda \theta_t\sigma'
        (\vb{\uv}{k}- \vb{{\uv^\star}}{k}  )^\top  A^\top A  (\vb{{\uv^\star}}{k} - \vv{\zv}{k})        
\\=&\ 
    \underbrace{\psi_k(\vb{{\uv^\star}}{k})
                  +\frac1K f(A\yv)  
                  + \nabla f(A \yv)^\top A (  \vb{{\uv^\star}}{k}- \vv{\yv}{k}  ) 
                  +\frac{\lambda \theta_t\sigma'}{2}\left\| A (\vb{{\uv^\star}}{k}- \vv{\zv}{k} )\right\|^2
               }_{= \G_k( \vv{{\uv^\star}}{k} ; \yv,\zv)}
\\&\qquad 
   + \ve{\vv{\vvv}{k}
                        +A^\top\nabla f(A \yv)
                        +{\lambda\theta_t\sigma' }A^\top A  (\vb{{\uv^\star}}{k} - \vv{\zv}{k})
        }{ \vv{\uv}{k} - \vb{{\uv^\star}}{k}  }
\\&\qquad 
   + \frac{ \lambda\theta_t\sigma' }{2}
        \left\| A (\vb{\uv}{k}- \vb{{\uv^\star}}{k}  )  \right\|^2   
\\ \geq &\  
    \G( \vv{{\uv^\star}}{k} ; \yv,\zv)    
   + \frac{ \lambda\theta_t\sigma' }{2}
    \left\| A (\vb{\uv}{k}- \vb{{\uv^\star}}{k}  )  \right\|^2,
\end{align*}
where the fact that $$\vv{\vvv}{k}
                        +A^\top\nabla f(A \yv)
                        +{\lambda\theta_t\sigma' }A^\top A  (\vb{{\uv^\star}}{k} - \vv{\zv}{k}) \in \partial \G_k( \vv{{\uv^\star}}{k} ; \yv,\zv)$$
yields the final inequality.
\end{proof}

We now present our main convergence theorem.  Our approach for proving the result is based on the use of randomized estimated sequences, as in~\cite{fercoq2015accelerated}.  However, we have included an important improvement to this approach that allows us to consider subproblems that are not strongly convex.  This is only possible due to the special structure of our objective function $\OA$.

\begin{theorem}\label{thm:convexCaseComplexity}
  For any optimal solution $\alpha_\star$ of problem~\eqref{eq:A} and all $t \geq 1$,
  \begin{equation}\label{asfsafsafasfa} 
    \begin{aligned}
   &\ \Exp[  \OA(\alphav_t) - \OA(\alphav_\star) ] \\
 \leq&\ \frac{4 }{(t\gamma-\gamma  +2 )^2} 
 \left(
   (1-\gamma )
   (\OA(\alphav_0) - \OA(\alphav_\star)) +  \frac{   \gamma \lambda  \sigma'}{2} C
+   K\epsilon_0\gamma+\sum_{j=1}^{t-1} E_j 
\right)
 \end{aligned}
 \end{equation} 
 holds where
 \begin{subequations}
 \begin{align}
   C &:= \sum_{k=1}^K
  \| A(\vv{\alphav_\star}{k}- \vf{\alphav}{k}{0} )\|^2,
   \\ 
    E_j &:= \frac{K \epsilon_j \gamma}{\theta_j}
 +  \frac{\gamma}{ \theta_{j-1}} \epsilon_{j-1} 
  + \gamma\cR  
     \sqrt{ 
\frac{2\lambda  \sigma' }
      { \theta_{j-1}} \epsilon_{j-1} }\ \ \text{for all}\ \ j \in \{1,\dots,t-1\}, \\
   \text{and}\ \ 
   \cR &:=  
   \max_{k,t}   \|  A 
  (\vv{\alphav_\star}{k}- \vb{\zv_t^\star}{k}
  )\|. \label{eq:defRR}
  \end{align}
\end{subequations}




  
 \end{theorem}
\begin{proof}
  By convexity of $g$, it follows that, for all $t \geq 0$, one has
\begin{equation}\label{eq:44}
g(\alphav_t)\overset{\eqref{eq:35}}{=} 
g\left(\sum_{l=0}^t \rho_t^l \zv_l\right)
\leq 
\sum_{l=0}^t \rho_t^l 
 g\left( \zv_l\right)
=:\hat g^t.  
\end{equation}
Combining this definition and the result from Lemma \ref{lemma2}, one gets that 
\begin{align}\label{eq:45}
\hat g^{t+1}  
&\overset{\eqref{eq:44},\eqref{eq:36}}{=}
\sum_{l=0}^{t+1} \rho_{t+1}^l 
 g\left( \zv_l\right)
= 
\gamma \theta_t g\left( \zv_{t+1}\right)
+
\sum_{l=0}^{t} \rho_{t+1}^l 
 g\left( \zv_l\right).
\end{align}
On the other hand, under Assumption~\ref{assumption.main}, one finds that
\begin{align*} 
&\ f(A\alphav_{t+1}) \\
\overset{\eqref{afasfasfdsafa}}{=}&\  
f(A\yv_t+ \gamma\theta_t A(  \zv_{t+1}-\zv_t)) \\
\leq&\  
   f(A\yv_t)
   + \left\langle \nabla f(A \yv_t), 
          \gamma\theta_t A (  \zv_{t+1}-\zv_t) \right\rangle 
   + \frac{\lambda\gamma^2\theta_t^2}{2}
   \left \| A(\zv_{t+1}-\zv_t) 
   \right\|^2 
\notag
\\
\overset{\eqref{eq:sigmaineq}}{\leq}&\  
   f(A\yv_t)
   + \gamma\theta_t\left\langle A^\top \nabla f(A \yv_t), 
          \sum_{k=1}^K(\vf{  \zv}{k}{t+1} - \vf{\zv}{k}{t}) \right\rangle 
   + \frac{ \lambda\gamma\sigma'\theta_t^2}{2}\sum_{k=1}^K
   \left \| A(\vf{  \zv}{k}{t+1} - \vf{\zv}{k}{t}) 
   \right\|^2 
\notag
\\
=&\ \sum_{k=1}^K \Bigg\{\frac{1}{K}(1-\gamma\theta_t) f(A\yv_t) 
 + \gamma\theta_t \left\langle A^\top\nabla f(A \yv_t),  \vf{\yv}{k}{t} - \vf{\zv}{k}{t}\right\rangle 
\\ 
 &\quad + \gamma\theta_t \Big(\frac{1}{K}f(A\yv_t) 
 +\left\langle A^\top \nabla f(A \yv_t), \vf{  \zv}{k}{t+1}-\vf{\yv}{k}{t}\right\rangle
 + \frac{ \lambda\sigma'\theta_t}{2}
  \left\| A(\vf{  \zv}{k}{t+1} - \vf{\zv}{k}{t}) \right\|^2\Big)  \Bigg\}.
\end{align*}
Next, note that from the definition of $\yv_t$ in the algorithm
one finds
\begin{align*}
        \vf{\yv}{k}{t}
         &\overset{\eqref{afwgrvwfeagwefrw}}{=} (1-\gamma\theta_t) \vf{\alphav}{k}{t} + \gamma\theta_t \vf{\zv}{k}{t},
\\
       \gamma\theta_t \vf{\yv}{k}{t} - \gamma\theta_t \vf{\zv}{k}{t}
         &=  (1-\gamma\theta_t) \vf{\alphav}{k}{t}
          - (1-\gamma\theta_t) \vf{\yv}{k}{t},
\\
       \gamma\theta_t (\vf{\yv}{k}{t} -   \vf{\zv}{k}{t})
         & =  (1-\gamma\theta_t) (\vf{\alphav}{k}{t}
          -   \vf{\yv}{k}{t}).
          \tagthis             \label{eq:47}       
\end{align*}
Defining, for all $t \geq 0$, an upper-bound on $\OA(\alphav_t)$ as \begin{equation}\label{eq:48}
 \hat\OA^t  :=
 \hat g^t + f(A\alphav_t)   
 \overset{\eqref{eq:44}}{\geq} \OA(\alphav_t),
\end{equation}
it follows from above and convexity of $f$ that one has
\begin{align*} 
 &\ \hat \OA^{t+1} \\
 \overset{\eqref{eq:48}}{=}&\   \hat g^{t+1}  +  f(A\alphav_{t+1})  
\\
\overset{\eqref{eq:45}}{=}&\  
 \gamma \theta_t g\left( \zv_{t+1}\right)
+
\sum_{l=0}^{t} \rho_{t+1}^l 
 g\left( \zv_l\right)
 +
 f(A\alphav_{t+1}) 
\\
\leq&\  
 \gamma \theta_t g\left( \zv_{t+1}\right)
+
\sum_{l=0}^{t} \rho_{t+1}^l 
 g\left( \zv_l\right)
 \\&\quad +
 \sum_{k=1}^K \Bigg\{\frac{1}{K}(1-\gamma\theta_t) f(A\yv_t) \\
 &\qquad +\gamma\theta_t \left\langle A^\top \nabla f(A \yv_t), \vf{\yv}{k}{t} - \vf{\zv}{k}{t}\right\rangle  
\\ 
    &\qquad  + \gamma\theta_t \Bigg(\frac{1}{K}f(A \yv_t)+\left\langle A^\top \nabla f(A \yv_t), \vf{  \zv}{k}{t+1}-\vf{\yv}{k}{t}\right\rangle
\\ &\qquad\qquad\qquad + \frac{ \lambda\sigma'\theta_t}{2}
  \left\| A(\vf{  \zv}{k}{t+1} - \vf{\zv}{k}{t}) \right\|^2\Bigg)  \Bigg\}
\\
\overset{\eqref{eq:36},\eqref{eq:47}}{=}&\  
 \gamma \theta_t g\left( \zv_{t+1}\right)
+
(1-\gamma\theta_t)
 \underbrace{\sum_{l=0}^{t}  \rho_{t}^l g\left( \zv_l\right)}_{\hat g^t} 
 \\&\quad +
 \sum_{k=1}^K \Bigg\{\frac{1}{K}(1-\gamma\theta_t) f(A\yv_t) \\
 &\qquad + (1-\gamma\theta_t) \left\langle A^\top \nabla f(A \yv_t), \vf{\alphav}{k}{t} - \vf{\yv}{k}{t}\right\rangle  
\\ 
    &\qquad  + \gamma\theta_t \Bigg(\frac{1}{K}f(A\yv_t)+\left\langle A^\top \nabla f(A \yv_t), \vf{  \zv}{k}{t+1}-\vf{\yv}{k}{t}\right\rangle
\\ &\qquad\qquad\qquad + \frac{ \lambda\sigma'\theta_t}{2}
  \left\| A(\vf{  \zv}{k}{t+1} - \vf{\zv}{k}{t}) \right\|^2\Bigg)  \Bigg\}
\\
\overset{\eqref{asfwfrwafwafas}}{=}&\ 
  (1-\gamma\theta_t)
 \hat g^t  
 +
    (1-\gamma\theta_t) 
    \left( f(A\yv_t) +  \left\langle A^\top \nabla f(A \yv_t), \alphav_{t} - \yv_{t}\right  \rangle  \right)
\\ 
    &\quad  + 
  \gamma\theta_t  
  \sum_{k=1}^K  
  \Bigg(
 \psi_k(\vv{\zv_{t+1}}{k}) +   
  \frac{1}{K}f(A\yv_t) +\left\langle A^\top \nabla f(A \yv_t), \vf{  \zv}{k}{t+1}-\vf{\yv}{k}{t}\right\rangle
 \\ &\qquad\qquad\qquad + \frac{ \lambda\sigma'\theta_t}{2}
  \| A(\vf{\zv}{k}{t+1} - \vf{\zv}{k}{t}) \|^2\Bigg)  
\\
\overset{\eqref{eq:subpdef4}}{\leq}&\ 
  (1-\gamma\theta_t)
 \left( \hat g^t  
 + f(A\alphav_t) 
  \right)
 + 
  \gamma\theta_t  
  \sum_{k=1}^K  
 \G_k(  \vf{\zv}{k}{t+1}  ; \yv_t, \zv_t ) 
\\
\overset{\eqref{eq:48}}{=}&\ 
  (1-\gamma\theta_t)
  \hat \OA^t
 + 
  \gamma\theta_t  
  \sum_{k=1}^K  
 \G_k(  \vf{\zv}{k}{t+1}  ; \yv_t, \zv_t ).     
  \end{align*}
Conditioning on the history up to time $t$, it follows from above that
\begin{align*} 
 \Exp_{t+1} [ \hat \OA^{t+1} | t ]  
 &\leq  
   (1-\gamma\theta_t)
  \hat \OA^t
 + 
  \gamma\theta_t  
  \sum_{k=1}^K  
 \Exp_{t+1} [ \G_k(  \vv{\zv_{t+1} }{k}  ; \yv_t, \zv_t )| t ]   
\\
&\overset{\eqref{eq:subprostate}}{\leq}      
   (1-\gamma\theta_t)
  \hat \OA^t
 + 
  \gamma\theta_t  
  \sum_{k=1}^K  
  \left(
   \G_k(  \vb{ \zv_{t+1}^\star  }{k}  ; \yv_t, \zv_t )    
  + \epsilon_t \right),
\end{align*}
meaning that, for any $\uv$, one has
\begin{align*}
&\ \Exp_{t+1} [ \hat \OA^{t+1} | t ] \\
\overset{\eqref{asdfsafdsafsfsa}}{\leq}&\ 
   (1-\gamma\theta_t)
  \hat \OA^t
 + 
  \gamma\theta_t  
  \sum_{k=1}^K  
  \left(
   \G_k(  \vb{\uv}{k}  ; \yv_t, \zv_t) 
-
\frac{\lambda \theta_t\sigma'}{2} 
\left\| A(
 \vb{\uv}{k} - \vb{  \zv_{t+1}^\star  }{k}   ) \right\|^2 
  + \epsilon_t \right). 
\end{align*}
In particular, choosing $\uv = \alphav_\star$, where $\alphav_\star$ is any optimal solution of problem~\eqref{eq:A}, and taking the total expectation, one finds
\begin{align*} 
 &\ \Exp  [ \hat \OA^{t+1}   ]  \\
 \leq&\   
   (1-\gamma\theta_t)
  \Exp[\hat \OA^t]
 + 
  \gamma\theta_t  
  \sum_{k=1}^K  
  \Exp\left[
   \G_k(  \vv{\alphav_\star}{k}  ; \yv_t, \zv_t ) 
-
\frac{\lambda \theta_t\sigma'}{2} 
\left\| A(
 \vv{\alphav_\star}{k} - \vb{  \zv_{t+1}^\star  }{k}   ) \right\|^2 
  + \epsilon_t \right]  
\\
=&\ 
   (1-\gamma\theta_t)
  \Exp[\hat \OA^t]
\\&\qquad
 + 
  \gamma\theta_t  
  \sum_{k=1}^K  
  \Exp\left[
  \frac{\lambda \theta_t\sigma'}{2}\|  A (\vv{\alphav_\star}{k}- \vf{\zv}{k}{t} )\|^2
-
\frac{\lambda \theta_t\sigma'}{2} 
\left\| A(
 \vv{\alphav_\star}{k} - \vb{  \zv_{t+1}^\star  }{k}   ) \right\|^2 
  + \epsilon_t \right]    
\\&\qquad
 + 
  \gamma\theta_t  
  \sum_{k=1}^K  
  \Exp\left[
   \psi_k( \vv{\alphav_\star}{k} ) 
  +\frac1K f( A{y}_t) 
  +\ve{ A^\top \nabla f(A \yv_t)}{  \vv{\alphav_\star}{k} - \vf{\yv}{k}{t}}
     \right]
\\
\leq&\ 
   (1-\gamma\theta_t)
  \Exp[\hat \OA^t]
  + 
  \gamma\theta_t  
    \OA(  \alphav_\star  ) 
\\&\qquad
 + 
  \gamma\theta_t  
  \sum_{k=1}^K  
  \Exp\left[
  \frac{\lambda \theta_t\sigma'}{2}\| A (\vv{\alphav_\star}{k}- \vf{\zv}{k}{t} )\|^2
-
\frac{\lambda \theta_t\sigma'}{2} 
\left\| A(
 \vv{\alphav_\star}{k} - \vb{  \zv_{t+1}^\star  }{k}   ) \right\|^2 
  + \epsilon_t \right],
 \end{align*}
where the last inequality follows from convexity of $f$.  Defining the scalar quantity $r^2_{t+1} =\sum_{k=1}^K \| A(
 \vv{\alphav_\star}{k} - \vb{  \zv_{t+1}^\star  }{k}   )  \|^2$, we may conclude from above that
\begin{align*} 
  &\ \Exp  \left[ \hat \OA^{t+1}-    \OA(  \alphav_\star  )  
+
\frac{\gamma \lambda \theta_t^2\sigma'}{2} 
r^2_{t+1} 
  \right] \\  
\leq&\ 
   (1-\gamma\theta_t)
  \Exp[\hat \OA^t - \OA(  \alphav_\star  )] 
  +K \epsilon_t \gamma\theta_t +   
  \frac{\gamma \lambda \theta_t^2\sigma'}{2}
  \Exp\left[
  \sum_{k=1}^K
  \| A (\vv{\alphav_\star}{k}- \vf{\zv}{k}{t} )\|^2
   \right].    \tagthis \label{asdfsafsafa}
 \end{align*}
To bound the last term on the right-hand side, observe that, for all $k \in \{1,\dots,K\}$,
\begin{align*}
\Exp
  \left[
  \|  A (\vv{\alphav_\star}{k}- \vf{\zv}{k}{t} )\|^2
   \right]
 &=  
\Exp\left[
  \|  A 
  (\vv{\alphav_\star}{k}- \vb{\zv_t^\star}{k}
  + \vb{\zv_t^\star}{k}  -\vf{\zv}{k}{t} 
  )\|^2 
   \right]
\\
 &\leq   
\Exp [\|  A 
  (\vv{\alphav_\star}{k}- \vb{\zv_t^\star}{k}  )\|^2]
+\underbrace{\Exp \left[
  \| A 
  (  \vb{\zv_t^\star}{k}  -\vf{\zv}{k}{t} 
  )\|^2   \right]}_{=:\ConstantA}   
  \\&\qquad +    
\underbrace{2 \Exp \left[
 \|  A 
  (\vv{\alphav_\star}{k}- \vb{\zv_t^\star}{k}
  )\| \cdot
  \| A 
  (
    \vb{\zv_t^\star}{k}  -\vf{\zv}{k}{t} 
  )\|     
   \right]}_{=:\ConstantB}
\\
 &\leq   
\Exp [r_t^2]
+\ConstantA + \ConstantB.   
\end{align*}

It remains to bound $\ConstantA$ and $\ConstantB$.
From \eqref{asdfsafdsafsfsa} and \eqref{eq:subprostate}, one finds that
\begin{equation}
  \begin{aligned}
    &\ \G_k(\vv{{\zv_{t+1}^\star}}{k}; \yv_t, \zv_t)
+\frac{\lambda \theta_t\sigma'}{2} 
\Exp_{t+1}[\| A 
  (
    \vb{\zv_{t+1}^\star}{k}  -\vf{\zv}{k}{t+1} 
  )\|^2 | t] \\
 \overset{\eqref{asdfsafdsafsfsa}}{\leq}&\ 
 \Exp_{t+1}[\G_k(\vv{\zv_{t+1}}{k} ; \yv_t, \zv_t) |t] \\
 \overset{\eqref{eq:subprostate}}{\leq}&\ \G_k(\vv{{\zv_{t+1}^\star}}{k}; \yv_t, \zv_t) + \epsilon_t,
  \end{aligned}
\end{equation}
 and hence we can conclude that
\begin{equation}
\label{asffasfdsafasfda}
    \frac{\lambda \theta_t\sigma'}{2}
    \Exp_{t+1}[ \| A(  \vb{\zv_{t+1}^\star}{k}  -\vf{\zv}{k}{t+1} ) \|^2|t] \leq \epsilon_t.
\end{equation}
Therefore,
\begin{align*}
\frac{\gamma\lambda \theta_t^{2}\sigma'}{2} \ConstantA&=
\frac{\gamma\lambda \theta_t^2\sigma'}{2} \Exp \left[
  \| A
  (  \vb{\zv_t^\star}{k}  -\vf{\zv}{k}{t} 
  )\|^2   \right]
\\
&=
\frac{\gamma\theta_t^2}{\theta_{t-1}}
 \frac{\lambda \theta_{t-1}\sigma'}{2} 
  \Exp_{t-1}   
\left[
\Exp_{t} \left[
\| A(  \vb{\zv_{t}^\star}{k}  -\vf{\zv}{k}{t} ) \|^2
| t-1
\right] \right]
\overset{\eqref{asffasfdsafasfda}}{\leq} \frac{\gamma\theta_t^2}{\theta_{t-1}} \epsilon_{t-1}.
\tagthis
\label{asfsafsafa}
\end{align*}
Now, let us bound $\ConstantB$. By defining $\cR$ as in \eqref{eq:defRR}, Jensen's inequality gives
\begin{align*}
 \frac{\gamma\lambda \theta_{t}^{2}\sigma'}{2} \ConstantB&=
  \frac{\gamma\lambda \theta_{t}^2\sigma'}{2} \left(2\Exp \left[
 \|  A 
  (\vv{\alphav_\star}{k}- \vb{\zv_t^\star}{k}
  )\| \cdot
  \|  A 
  (
    \vb{\zv_t^\star}{k}  -\vf{\zv}{k}{t} 
  )\|     
   \right]\right)
\\
&\overset{\eqref{eq:defRR}}{\leq}
\cR 
     \gamma\lambda \theta_{t}^2\sigma' \Exp \left[
   \|  A 
  (
    \vb{\zv_t^\star}{k}  -\vf{\zv}{k}{t} 
  )\|     
   \right]
\\
&\leq   
\cR 
     \gamma\lambda \theta_{t}^2\sigma'
     \sqrt{ \frac{2} {\gamma\lambda \theta_t^2\sigma'} }
     \sqrt{
\frac{\gamma\lambda \theta_t^2\sigma'}{2}      
     \Exp \left[
   \|  A 
  (
    \vb{\zv_t^\star}{k}  -\vf{\zv}{k}{t} 
  )\|^2     
   \right] }  
\\
&\overset{\eqref{asfsafsafa}
}{\leq}   
\cR 
    \gamma\lambda \theta_{t}^2\sigma'
     \sqrt{ \frac{2} {\gamma\lambda \theta_t^2\sigma'} }
     \sqrt{
     \frac{\gamma\theta_t^2}{\theta_{t-1}} \epsilon_{t-1} }     
=
\gamma\theta^2_{t}\cR  
     \sqrt{ 
\frac{2\lambda  \sigma' }{ \theta_{t-1}} \epsilon_{t-1} }.     
\end{align*}
Putting everything together leads to  
\begin{align*} 
 \Exp  \left[ \hat \OA^{t+1}-    \OA(  \alphav_\star  )  
+
\frac{\gamma \lambda \theta_t^2\sigma'}{2} 
r^2_{t+1} 
  \right]  
&\overset{\eqref{asdfsafsafa}}{\leq}
  \Exp\left[(1-\gamma\theta_t)(\hat \OA^t - \OA(  \alphav_\star  ))+  \frac{\gamma \lambda \theta_t^2\sigma'}{2} r_t^2
\right] 
\\&\qquad
  +K \epsilon_t \gamma\theta_t
 +  \frac{\gamma\theta_t^2}{\theta_{t-1}} \epsilon_{t-1} 
 + \gamma\theta^2_{t}\cR  
     \sqrt{ 
\frac{2\lambda \sigma' }{ \theta_{t-1}} \epsilon_{t-1} }.
 \end{align*}
Dividing both sides by 
$\theta_t^2$ 
and denoting $\phi_t := \Exp[\hat \OA^t - \OA(  \alphav_\star  )]$
and $\tilde r_t^2 := \Exp[r_t^2]$ gives
\begin{align*} 
 \frac1{\theta_t^2} \phi_{t+1} 
+
\frac{\gamma \lambda  \sigma'}{2} 
\tilde r^2_{t+1}  
&\overset{\eqref{asdfsafsafa}}{\leq}
   \frac{1-\gamma\theta_t}{\theta_t^2}
   \phi_t
  +   \frac{\gamma \lambda  \sigma'}{2} \tilde r_t^2
\\&\qquad
  +
\underbrace{  \frac{K \epsilon_t \gamma}{\theta_t}
 +  \frac{\gamma}{ \theta_{t-1}} \epsilon_{t-1} 
  + \gamma\cR  
     \sqrt{ 
\frac{2\lambda  \sigma' }
      { \theta_{t-1}} \epsilon_{t-1} }}_{=:E_t}.
      \tagthis \label{afsafsafasfda}
 \end{align*}
Now, by the property of $\theta_t$ in \eqref{eq:34}, one finds
\begin{align*} 
\frac{1-\gamma\theta_{t+1}}{\theta^2_{t+1}} \phi_{t+1} 
+
\frac{\gamma \lambda  \sigma'}{2} 
\tilde r^2_{t+1}  
\leq
   \frac{1-\gamma\theta_t}{\theta_t^2}
   \phi_t
  +   \frac{\gamma \lambda  \sigma'}{2} \tilde r_t^2
+ E_t.
\tagthis \label{asdfasfsafagvdsa}
 \end{align*}
Unrolling the recurrence, one obtains for $t\geq 1$ that
\begin{align*} 
\frac{1-\gamma\theta_{t}}{\theta^2_{t}} \phi_{t} 
+
\frac{\gamma \lambda  \sigma'}{2} 
\tilde r^2_{t}  
&\overset{\eqref{asdfasfsafagvdsa}}{\leq}
   \frac{1-\gamma\theta_0}{\theta_0^2}
   \phi_0
  +   \frac{\gamma \lambda  \sigma'}{2} \tilde r_0^2
+\sum_{i=0}^{t-1} E_{i} .
\tagthis \label{aseagvrefwdsaXXXX}
 \end{align*}
Hence, along with \eqref{eq:34}, one has for $t \geq 1$ that
\begin{align*} 
  \phi_{t}
&\leq
 \theta^2_{t-1}
 \left(
   \frac{1-\gamma\theta_0}{\theta_0^2}
   \phi_0
  +   \frac{   \gamma \lambda  \sigma'}{2} \tilde r_0^2
+  \sum_{i=0}^{t-1} E_{i}.
\right)
\\
&\overset{\eqref{afsdfasfadsfdsa}}
 {\leq}
 \left(\frac{2 }{t\gamma-\gamma  +2 }\right)^2
 \left(
   \frac{1-\gamma\theta_0}{\theta_0^2}
   \phi_0
  +   \frac{   \gamma \lambda  \sigma'}{2} \sum_{k=1}^K
  \|  A (\vv{\alphav_\star}{k}- \vf{\zv}{k}{0} )\|^2 
+   K\epsilon_0\gamma+\sum_{i=1}^{t-1} E_{i} 
\right)
 \end{align*}
and  \eqref{asfsafsafasfa} follows.
\end{proof}
 
Theorem~\ref{thm:convexCaseComplexity} describes the behavior of suboptimality only.  We can, however, use the following theorem from \cite{dunner2016primal}, which relates suboptimality with the duality gap. 
\begin{theorem}[{\cite[Theorem 4]{dunner2016primal}}]
\label{thm:pd}
Suppose problem \eqref{eq:A} is solved by a (possibly randomized) algorithm producing a sequence of iterates 
$\{\alphav_t\}_{t=0}^\infty$ such that, for all $t \geq 1$,
$$\Exp[\OA(\alphav_t) - \OA(\alphav_\star)] \leq \frac{F}{d(t)}$$ for some scalar $F \geq 0$ and function $d$.  If, for $t \geq 1$, it holds that
\begin{equation}\label{d-function}
  d(t) \geq \max 
\left\{ \frac{2F\lambda n  }{\tilde\sigma^2L^2}, 
   \frac{2F \tilde\sigma^2 L^2}{ \lambda n\epsilon^2}  \right\},
\end{equation}
where $\tilde \sigma^2$ is the maximum eigenvalue of $A^\top A$, then the expected duality gap satisfies 
$$\Exp[ \OA(\alphav_t) + \OB(\wv(\alphav_t))]\leq \epsilon.$$
\end{theorem}

Related to \eqref{d-function}, henceforth, we assume that $\epsilon$ is such that
$\frac{2F  \lambda n}{\tilde\sigma^2L^2} <
   \frac{2F \tilde\sigma^2 L^2}{ \lambda n \epsilon^2}$.
 
Before stating our key corollaries of Theorem~\ref{thm:convexCaseComplexity} and comparisons with results for other methods in the literature, let us define a few important quantities on which these results depend (see Table~\ref{asfasfsafas}). 
The first quantity is $\tilde \sigma^2$, already defined in Theorem~\ref{thm:pd}.  Due to the fact that for each data column $A_i$ we have $\|A_i\|\leq 1$
, it follows that this quantity is bounded by $n$.
The second quantity is $\sigma_k^2$, the maximum eigenvalue of a Gram matrix for the local data on node $k$.  This value will be large if the samples stored on node $k$ are correlated.  The next important quantity is $\sigma^2$, which depends on each $\sigma_k^2$ and the size of each partition $n_k$.  Let us now, for simplicity, assume that the data are split equally, i.e., $n_k = \frac{n}{K}$ for all $k \in [K]$.  The last quantity is $\cR$, which measures the maximum possible distance between the optimal solution and each iterate, i.e.,
 \begin{align*}
\cR \overset{\eqref{eq:defRR}}{:=}  
\max_{k,t}   \|  A 
  (\vv{\alphav_\star}{k}- \vb{\zv_t^\star}{k}
  )\| 
&\leq \max_{k,t}   
\sigma_k \|   
  (\vv{\alphav_\star}{k}- \vb{\zv_t^\star}{k}
  )\|    
\\
&\leq   
 \max_{k}   
\left(\sigma_k 2 L \sqrt{ n_k }\right)
\leq
 \max_{k}   
 \left({2L n_k}\right).  
\end{align*}
Here, we have used the fact that since $g_i^*$ is $L$-Lipschitz, it follows that $g_i$ has an effective domain bounded within a ball of radius $L$ about the origin, which in turn means that for all $(a,b) \in \dom(g_i) \times \dom(g_i)$ one has
$\|a-b\| \leq 2L$.

%
%

\begin{table*}
\caption{Important quantities for our comparisons with related convergence results.
}
\label{asfasfsafas}
\centering
\vspace{2mm}
\begin{tabular}{c|c|c|c}
 & & worst-case & worst-case \\
symbol & expression & upper bound & $\lim_{K \to n}$
\\ \hline \hline 
$\tilde\sigma^2$ & $\max_{\alphav} \frac{\|A \alphav\|^2}{\|\alphav\|^2}$ & $  n$  & $n$
\\

$\sigma_k^2$ & $\max_{\vv{\alphav}{k}} \frac{\|A \vv{\alphav}{k} \|^2}{\|\vv{\alphav}{k}\|^2}$ & $ n_k =\frac{n}{ K}$ &  $1$
\\
$\sigma^2$ & $\sum_{k=1}^K |\mathcal{P}_k| \sigma_k^2$ & $\sum_k n_k^2   = K \frac{n^2}{K^2} = \frac{n^2}{K}$ & $n$
\\
$\cR$ &  $\min_{\alphav_\star} \max_{k,t}   \|A 
  (\vv{\alphav_\star}{k}- \vb{\zv_t^\star}{k}
  )\|$
     & $ \max_{k}   
 {2Ln_k} = \frac{2Ln}{K}$
 & ${2L}$
\end{tabular}
\end{table*}

\subsection{Exact Subproblem Solves}
\label{sec:exact}

If the subproblems are solved exactly, i.e., if $\epsilon_t=0$ for all $t \geq 0$,   
 then \eqref{asfsafsafasfa}
has the form
\begin{align*} 
  \Exp[  \OA(\alphav_t) - \OA(\alphav_\star) ]
 \leq &
\tfrac{4 }{(t\gamma-\gamma  +2 )^2} 
 \Big(
   (1-\gamma)
      (\OA(\alphav_0) - \OA(\alphav_\star))\\
  & \qquad\qquad\quad + \tfrac{   \gamma \lambda  \sigma'}{2} \textstyle{\sum}_{k=1}^K
  \| A (\vv{\alphav_\star}{k}- \vf{\alphav}{k}{0} )\|^2
\Big).\label{asfdafdsafa}\tagthis
 \end{align*} 
A nice property of this result is that the second term in the parentheses might be equal to zero even if $\zv_0 = \alphav_0 \neq \alphav_\star$.  This is not the case for other results for accelerated algorithms as their subproblems are strongly convex~\cite{nesterov2013introductory,beck2009fast,
fercoq2015accelerated,lin2015universal} and hence have the term $\|\alphav_\star - \alphav_0\|^2$ present in their complexity guarantees.  Another nice property is that  
$\textstyle{\sum}_{k=1}^K
  \| A (\vv{\alphav_\star}{k}- \vf{\alphav}{k}{0} )\|^2
$
can be bounded above 
by 
$\textstyle{\sum}_{k=1}^K
\sigma_k^2  \| \vv{\alphav_\star}{k}- \vf{\alphav}{k}{0}  \|^2
$, though the former can be much smaller.

\begin{corollary}
\label{cor:exact}
Consider the extreme cases $\gamma=\frac1K$ and $\gamma=1$.
As discussed in Section~\ref{sec:acocoa},
for the first case, one can choose $\sigma'=1$ while for the second case one can choose $\sigma'=K$.
To obtain $\Exp[  \OA(\alphav_t) - \OA(\alphav_\star) ] \leq \epsilon$, one has to run Algorithm \ref{alg:algo1} for at least $t>T$ iterations, where $T$ is defined as follows for each case.
\begin{itemize}[leftmargin=0.5cm,topsep=0pt,itemsep=0ex,partopsep=0ex,parsep=0ex]
\item {\bf Case $\gamma=1$}: 
\[
T= 
\sqrt{ 
   \frac{  2   \lambda  \sigma'}{ \epsilon}\sum_{k=1}^K
  \|  A (\vv{\alphav_\star}{k}- \vf{\alphav}{k}{0} )\|^2
  }
\leq   
  \sqrt{ 
   \frac{  8 L^2      }{ \lambda \epsilon} 
  }
\]

\item {\bf Case $\gamma=1/K$}: 
\[
T\leq 
\sqrt{
\frac{4K(K-1)}
     {\epsilon} 
    (\OA(\alphav_0) - \OA(\alphav_\star))
+\frac{8L^2}{\lambda  \epsilon}
}
.
\]
\end{itemize}
\end{corollary}

As found in \cite{ma2015adding}, the complexity for $\gamma=1$ is better 
 as one requires fewer iterations in order to have an $\epsilon$-approximate solution in expectation.  Hence, let us focus only on the case of $\gamma=1$ and compare our rate with the results derived in 
\cite{ma2015adding}.
From the 
proof of \cite[Theorem 8]{ma2015adding},
one obtains that for \textsc{CoCoA+}, if $\alphav_0 = {\bf 0}$, it holds that 
$
\Exp[\OA(\alphav_t) - \OA(\alphav_\star)]
\leq \frac{ 4KL^2 \lambda\sigma^2  }{  (1+\frac12 (t -t_0  )) }
$,
where 
$$t_0 \geq  \max\{0, \log(2(\OA(\alphav_0) - \OA(\alphav_\star))/(4KL^2 \lambda\sigma  ))\}.$$
To obtain suboptimality below~$\epsilon$, \textsc{CoCoA+} needs to be run for $t\geq T = t_0 + \frac{ 8KL^2 \lambda\sigma^2  }{  \epsilon }
\leq 
t_0 + \frac{ 8 L^2   }{ \lambda   \epsilon }
$ iterations.
Neglecting the $t_0$ term, 
\textsc{CoCoA+} needs 
$\frac{ 8 L^2   }{ \lambda   \epsilon }$ iterations, whereas
\textsc{AccCoCoA+} needs only 
$\sqrt{ 
\frac{ 8 L^2 }{\lambda  \epsilon}
  }$.
This improvement is consistent with proximal gradient and accelerated gradient descent 
\cite{nesterov2013gradient}, which is as expected since, in the worst-case, they will produce the same iterates as \textsc{CoCoA+} and \textsc{AccCoCoA+}. 

Let us also derive a complexity bound for the duality gap. Corollary~\ref{cor:exact}
with 
Theorem~\ref{thm:pd}
implies that whenever 
$t \geq T$, where $T = \frac{2 L }
     {\epsilon}
\sqrt{ 
\frac{ \sigma'  \tilde\sigma^2  }{  n }      
\textstyle{\sum}_{k=1}^K
  \| A (\vv{\alphav_\star}{k}- \vf{\alphav}{k}{0} )\|^2
}
\leq 
\frac{4 L^2 }
     {\lambda \epsilon}$,
the expected duality gap satisfies $\Exp[ \OA(\alphav_t) + \OB(\wv(\alphav_t)) ]\leq \epsilon$. This is valid for the number of iterations of \textsc{AccCoCoA+}.  Note also that for the worst-case complexity for non-accelerated batch-SDCA (or full gradient method) for hinge loss ($L=1$), the known results
hold for the {\it average solution}
and the number of iterations are 
$\sim \mathcal{O}( \frac{1}{\lambda \epsilon})$ (see, e.g., \cite{takavc2013mini,takavc2015distributed}
with the batch size chosen as large as possible).  

\subsection{Inexact Subproblem Solvers}
\label{sec:inexact}

In order to get a better understanding of the case when the subproblems are solved approximately, let us define an auxiliary nonnegative sequence 
$\{a_t\}_{t=0}^\infty$ 
such that $\frac{ \sum_{t=0}^{\infty} \sqrt{a_t}}{t^2} \to 0$.
In this section, we assume that the errors for the local solvers are set as $\epsilon_t = a_t \theta_t$.

Before analyzing this case in more detail, let us bound the total accumulated error up to iteration $t$, i.e., $\sum_{j=1}^{t-1} E_j$.
One finds that 
\begin{align*}
\sum_{j=1}^{t-1} E_j
&=
K \gamma
\sum_{j=1}^{t-1} 
    a_j 
 + 
 \sum_{j=0}^{t-2}     a_{j} 
  +
 \cR \sqrt{2\lambda  \sigma'} 
  \sum_{j=1}^{t-1}  
     \sqrt{    
      a_{j-1} } \notag\\
&\leq 
  (K+1) \sum_{j=0}^{t-1} a_j 
+ 
 \cR \sqrt{2\lambda  \sigma'} 
  \sum_{j=0}^{t-1}  
     \sqrt{    
      a_{j} }
\\
&\leq       
  (K+1) \sum_{j=0}^{t-1} a_j 
+ 
  \sqrt{ \frac{ 8L^2 }
              { \lambda  K           }
  } 
  \sum_{j=0}^{t-1}  
     \sqrt{    
      a_{j} } =: S_t.
\end{align*}

Let us now consider two cases:
\begin{itemize}[leftmargin=0.5cm,topsep=0pt,itemsep=0ex,partopsep=0ex,parsep=0ex]
\item
Suppose that 
$a_t = r \in \R^+$. In this case, $\epsilon_t = r \theta_t = \mathcal{O}(1/t)$.
Moreover, 
$S_t =\left( (K+1)r + \sqrt{r \frac{ 8L^2 }
              { \lambda  K           }
  }\right)  t$. This implies that 
\begin{align*}
 \Exp[  \OA(\alphav_t) - \OA(\alphav_\star) ]
 \leq & 
  4\tfrac{
     \frac{    \lambda  \sigma'}{2} \textstyle{\sum}_{k=1}^K
  \|  A (\vv{\alphav_\star}{k}- \vf{\alphav}{k}{0} )\|^2
+   K\epsilon_0 +
S_t }{(t+1 )^2} \\
\sim&  
 \mathcal{O}
 \left(
  \tfrac{ (K+1)r + \sqrt{r \frac{ 8L^2 }
              { \lambda  K           }
  } }{t}  \right).
\end{align*}
Note that, in this case, $\epsilon \sim \mathcal{O}(1/t)$.  This might create the impression that the
local solver has to do more work as $t$ increases; however, note that the
Lipschitz constant of the gradient of the smooth part of the subproblem solved by the local solver 
also scales as $\theta_t \sim \mathcal{O}(1/t)$.

\item 
A second interesting case is when
$\lim_{t\to \infty} S_t =:  S_\infty<\infty$.  For example, suppose $a_t=\frac{r}{t^p}$
with $p>2$.  Then, indeed, $\lim_{t\to \infty} S_t$ is finite. In this case, one obtains
\begin{align*} 
 \Exp[  \OA(\alphav_t) - \OA(\alphav_\star) ]
&\leq  
  4\tfrac{
     \frac{    \lambda  \sigma'}{2} \textstyle{\sum}_{k=1}^K
  \|  A (\vv{\alphav_\star}{k}- \vf{\alphav}{k}{0} )\|^2
+   K\epsilon_0 +
S_t }{(t+1 )^2}
\\
&\leq 
  4\tfrac{
     \frac{    \lambda  \sigma'}{2} \textstyle{\sum}_{k=1}^K
  \|A (\vv{\alphav_\star}{k}- \vf{\alphav}{k}{0} )\|^2
+   K\epsilon_0 +
 S_\infty  }{(t+1 )^2}
\\
&\leq    
   \tfrac{ 
       \frac{ 8L^2}{\lambda   }
+  4 K\epsilon_0 +4
  (K+1) \sum_{i=0}^{\infty} a_i 
+ 
4  \sqrt{ \frac{ 8L^2 }
              { \lambda  K           }
  } 
  \sum_{i=0}^{\infty}  
     \sqrt{    
      a_{i} }
 }{(t+1 )^2} 
\sim \mathcal{O}(1/t^2).
 \end{align*}

\end{itemize}

\section{Numerical Experiments}

In this section, we report the results of numerical experiments.  The purpose of providing the results of these experiments is twofold.  For one thing, we use them to illustrate the benefits of acceleration by providing results that compare the performance of \textsc{AccCoCoA+} versus \textsc{CoCoA+} \cite{ma2015adding,Smith:2016wp}.  In addition, we explore the communication/computation tradeoff and the scalability of \textsc{AccCoCoA+}.  For a recent comparison of \textsc{CoCoA+} to other distributed solvers, including Quasi-Newton methods and ADMM, we refer the reader to~\cite{ma2015adding,Smith:2016wp}.

Our implementations of \textsc{CoCoA+} and \textsc{AccCoCoA+} are written in C++ using MPI for communication, run on \textsf{\small m3.xlarge} Amazon EC2 instances. 
We run all the experiments across $K=4$ nodes using SDCA \cite{shalev2013stochastic} as the local solver.  The datasets we used are summarized in Table \ref{tab:datasets}.\footnote{All data are available at \url{https://www.csie.ntu.edu.tw/~cjlin/libsvmtools/datasets/}.}

\begin{table}[ht]
 \caption{Datasets used in the numerical experiments.}
\label{tab:datasets}  
 \centering
{
      \begin{tabular}{crrr}
      \hline
    Dataset & \multicolumn{1}{c}{$n$} &
    \multicolumn{1}{c}{$d$} & 
    \multicolumn{1}{c}{size(GB)} \\
    \hline 
        url & 2,396,130 & 3,231,961 & 2.21\\
  rcv1.test & 677,399 & 47,236 & 1.21 \\
    covtype  & 581,012 &   54 & 0.07 \\ 
    epsilon  & 400,000 & 2,000 & 3.6 \\ 
     \hline  
      \end{tabular}
}    
 
\end{table}


We first compare \textsc{CoCoA+} versus \textsc{AccCoCoA+} for solving hinge-loss SVM problems of the form \eqref{eq:ERM}.  For both algorithms, the local solver, SDCA, is run for $H=5\times10^5$ iterations (closed-form single coordinate solutions).  In Figure~\ref{fig:f1}, we compare the evolution of the duality gap with respect to the number of iterations and elapsed time.  The results suggest a benefit of acceleration in terms of decreasing the duality gap, both when $\gamma = 1/K$ and when $\gamma = 1$.  In Figure~\ref{fig:f2},
we show how the regularization parameter $\lambda$ can affect the performance of both algorithms when solving the problem with the $\texttt{url}$ dataset.  In particular, the experiments suggest that as the value of $\lambda$ becomes smaller, there will be a more significant benefit from employing the accelerated algorithm.

\begin{figure}[ht]
\centering
  \includegraphics[width=4cm]{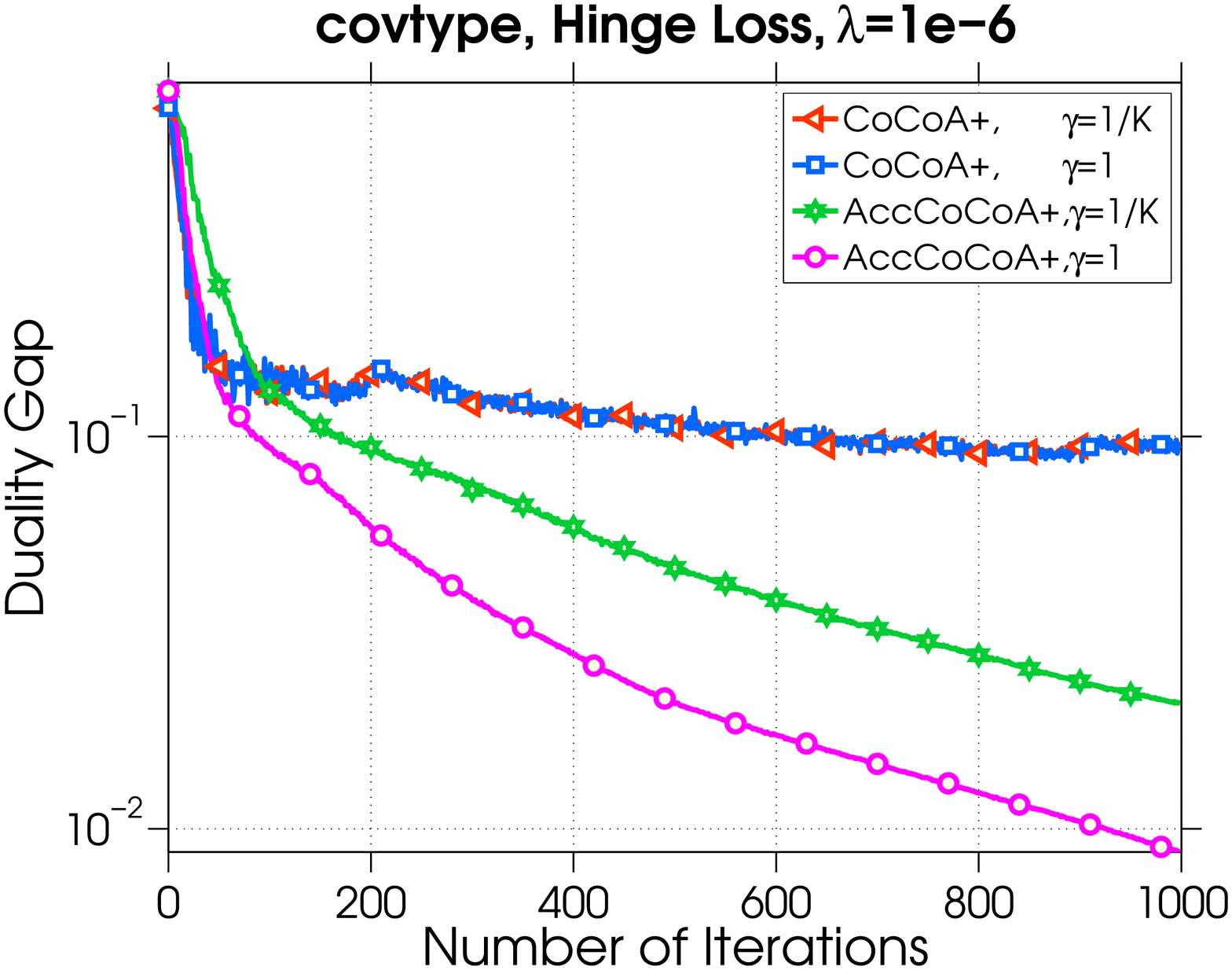}
  \includegraphics[width=4cm]{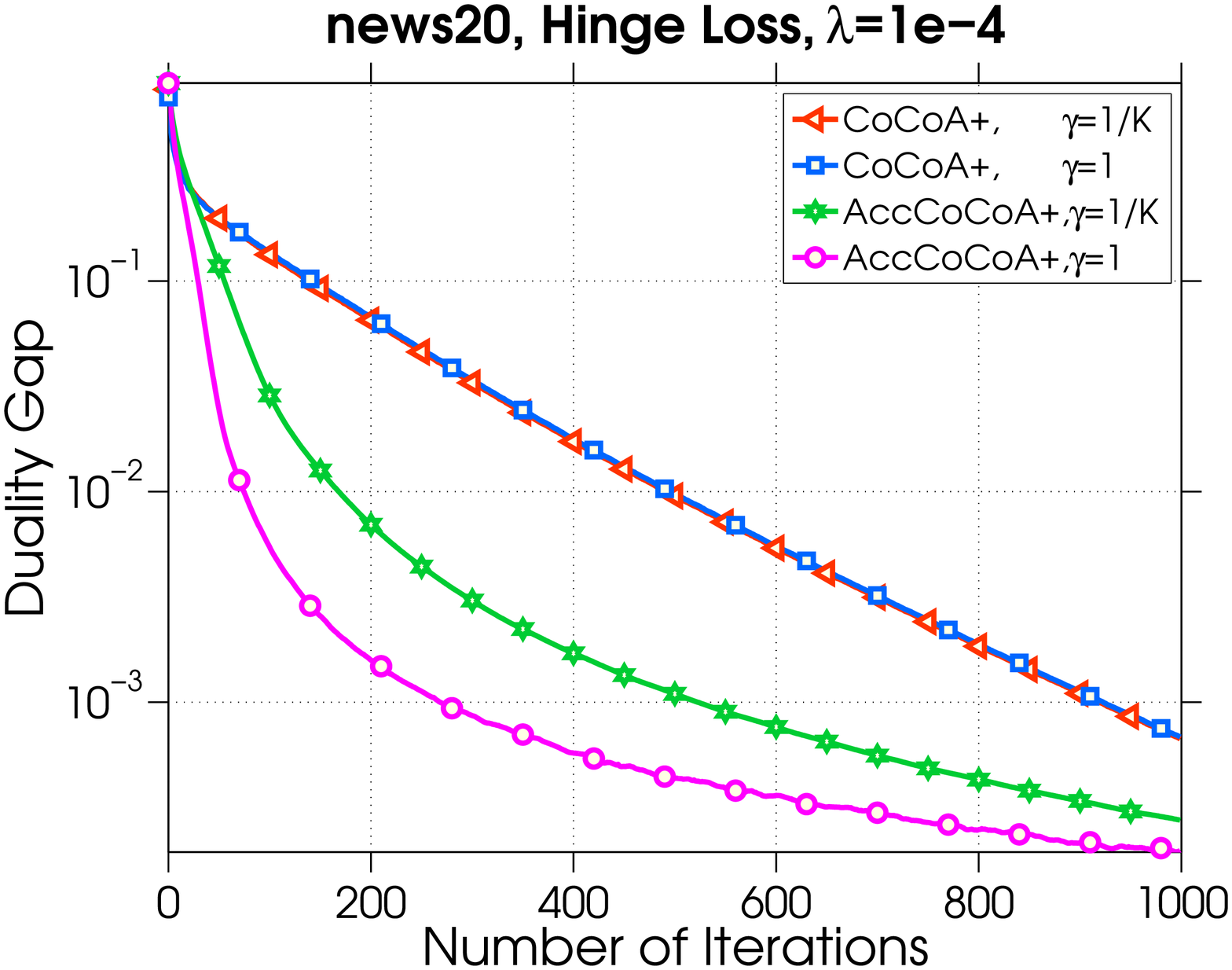}
  \includegraphics[width=4cm]{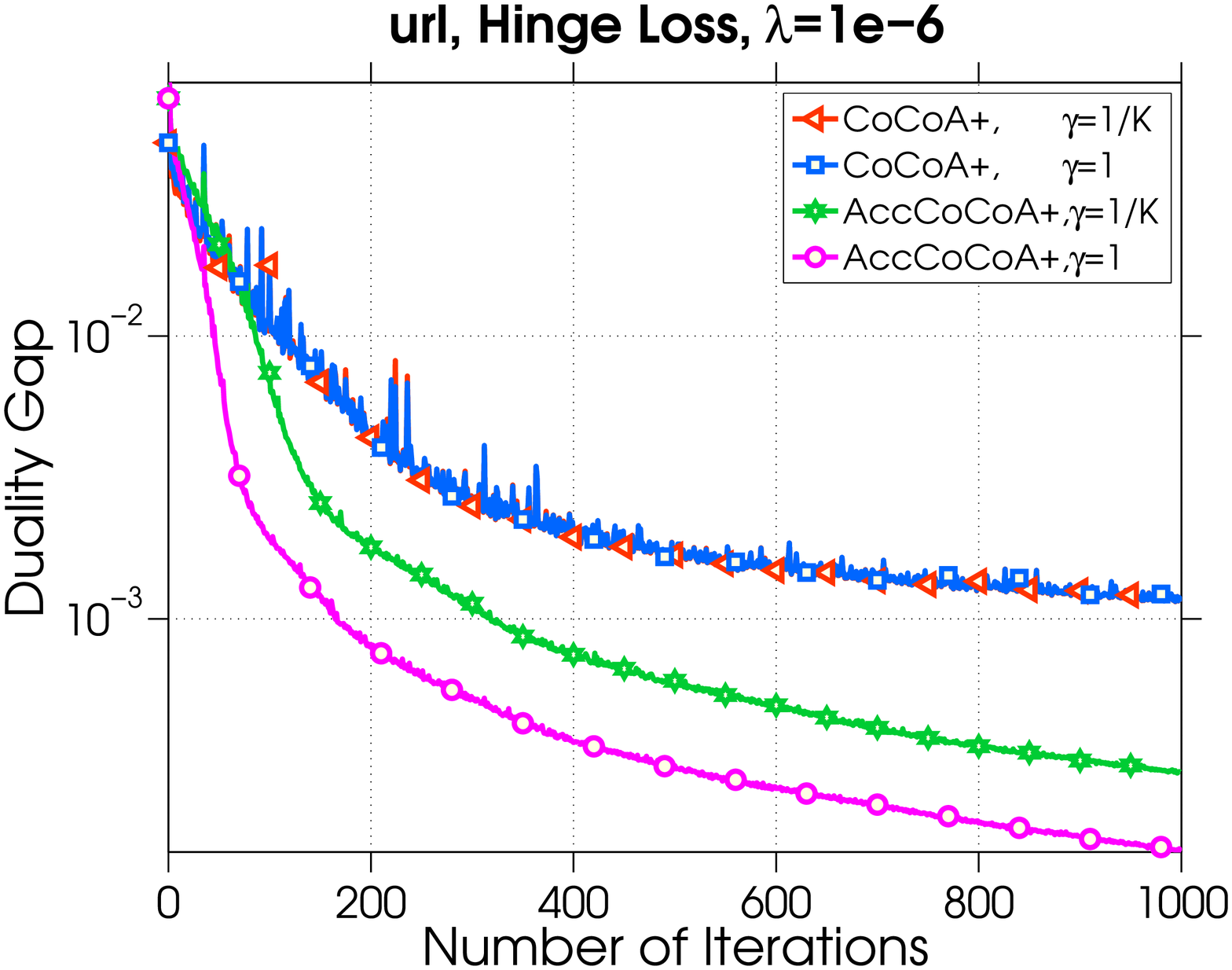}

  \includegraphics[width=4cm]{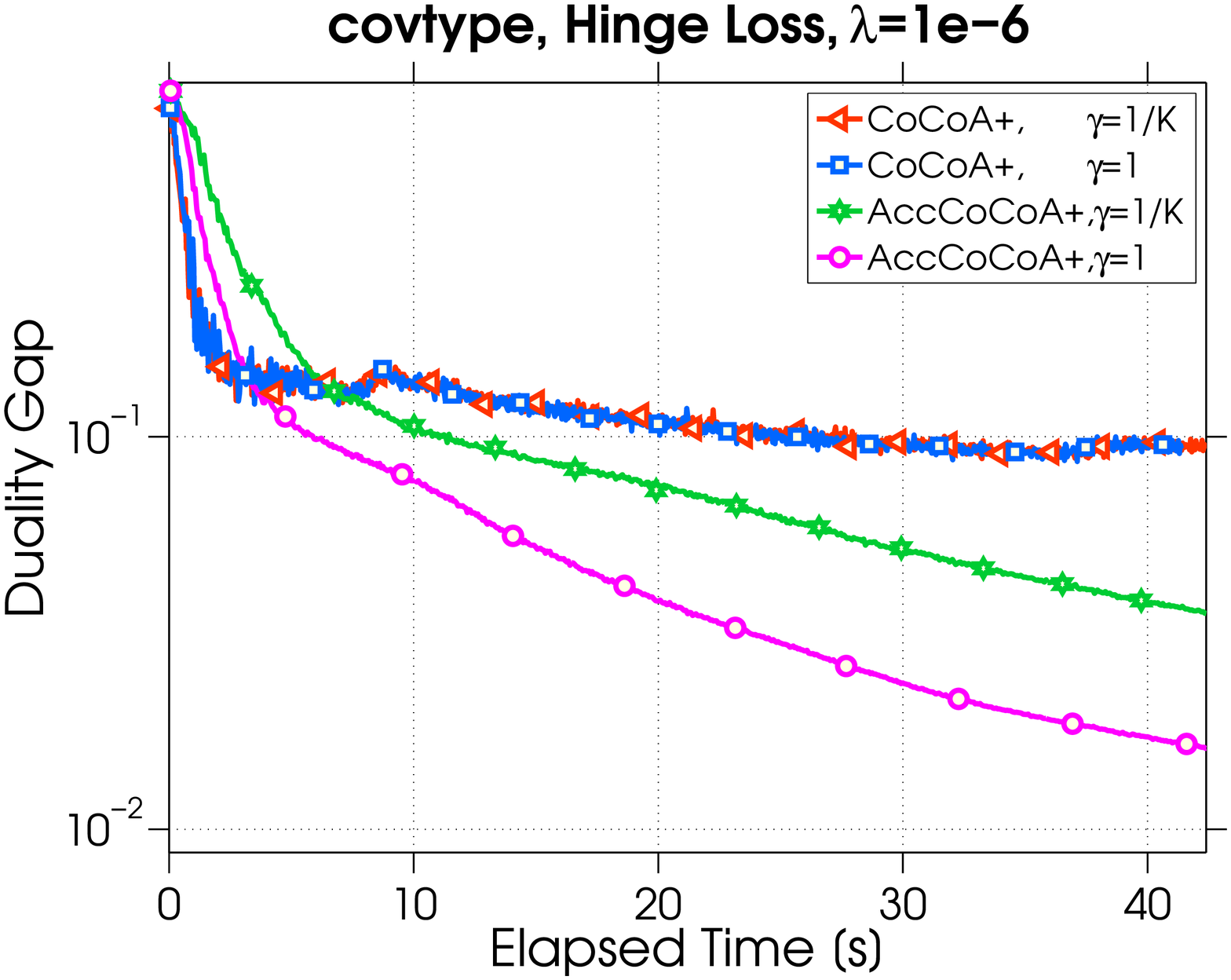}
  \includegraphics[width=4cm]{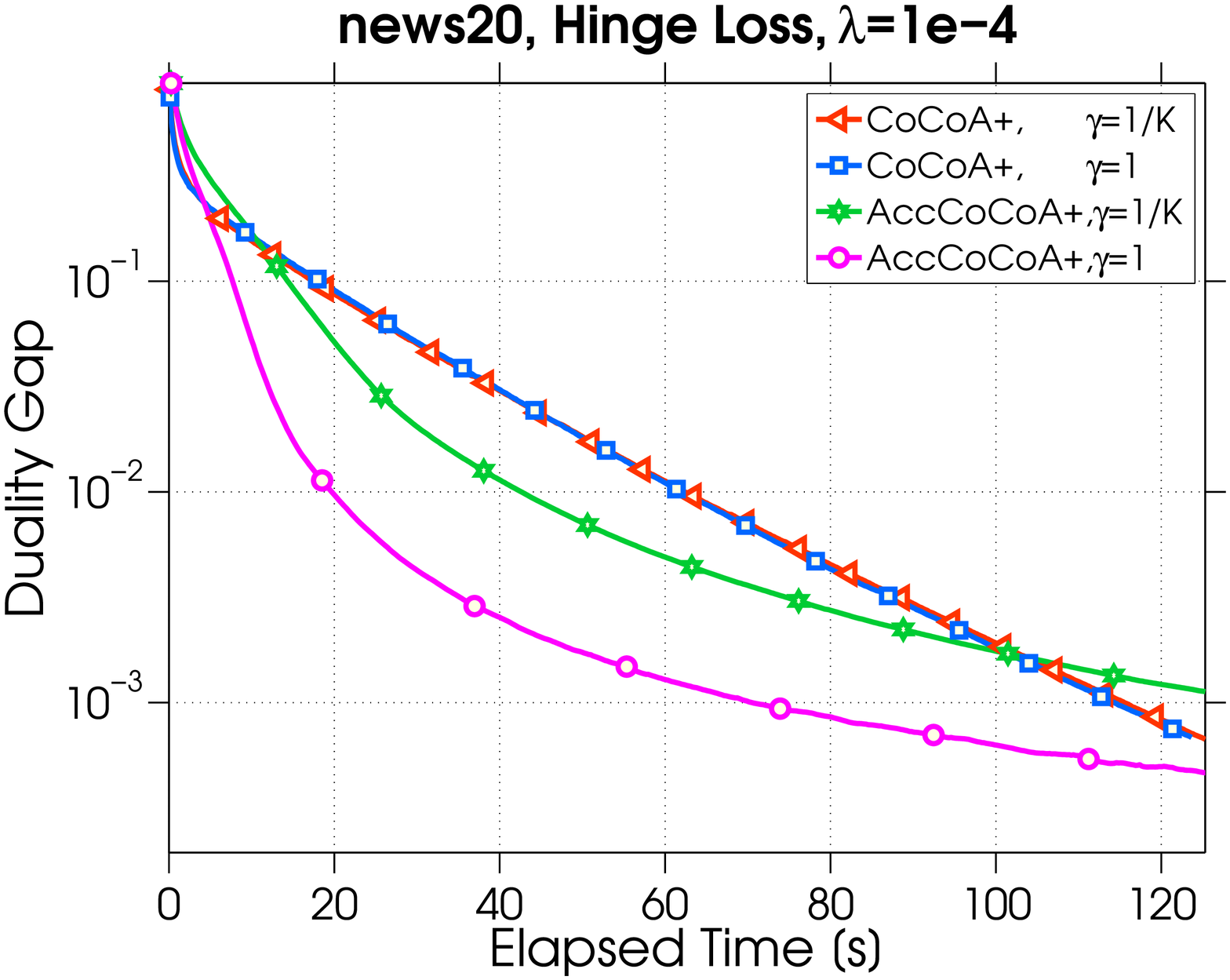}
  \includegraphics[width=4cm]{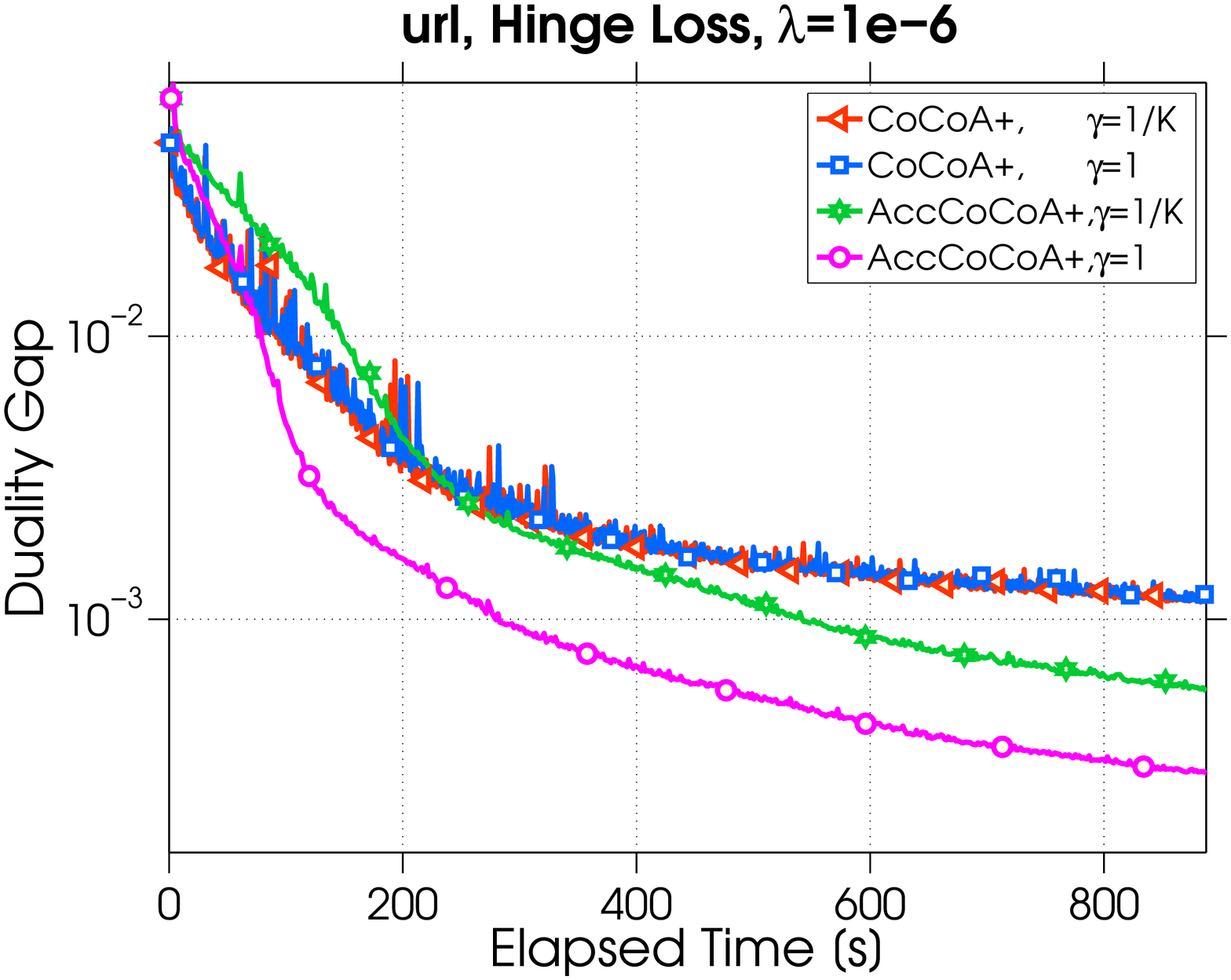}
\caption{Duality gap as a function of iterations (top row) and elapsed time (bottom row) when solving hinge-loss SVM problems.}
\label{fig:f1}
$\ $\\
 \includegraphics[width=4cm]{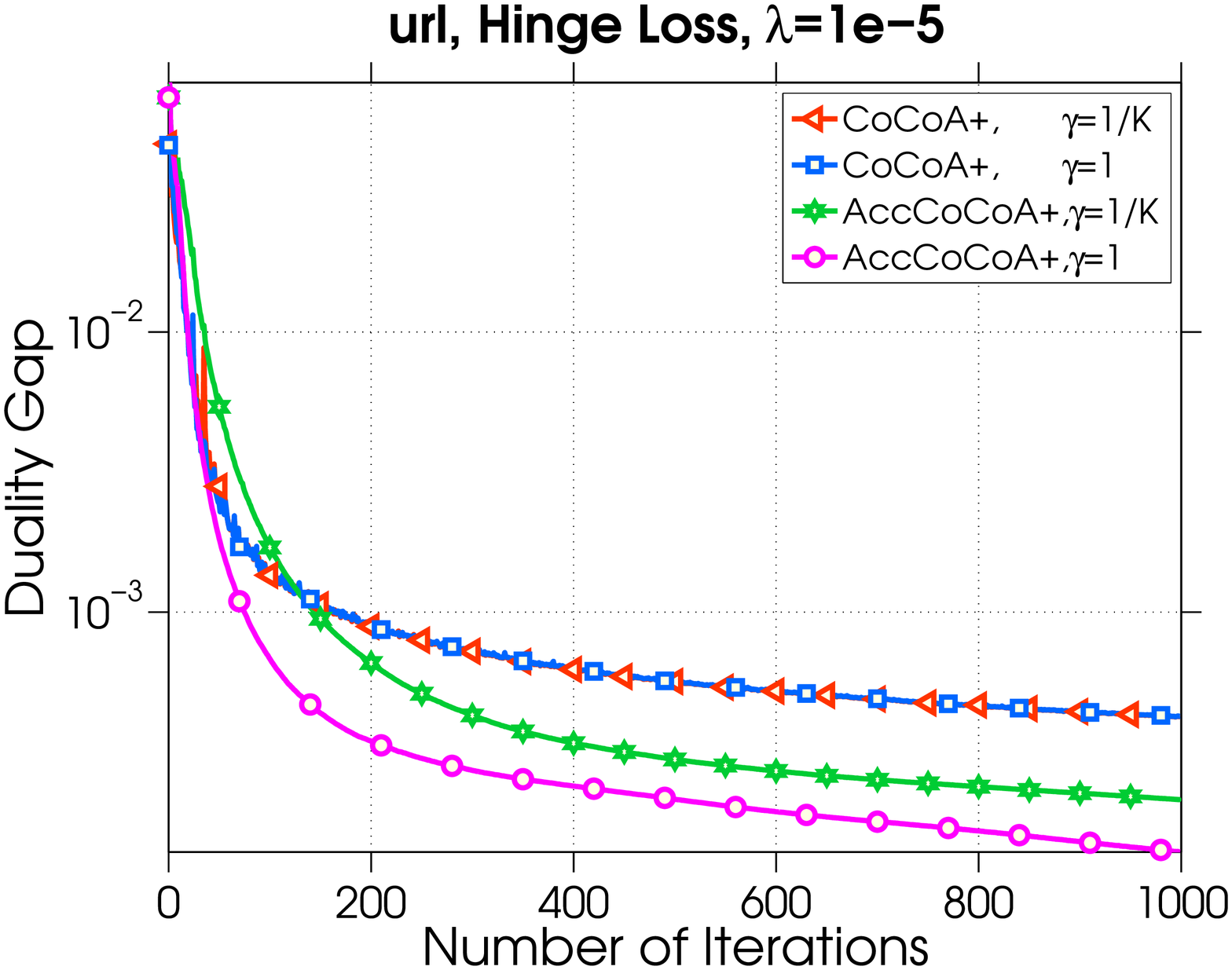}
  \includegraphics[width=4cm]{url_1e-6_1.eps}
  \includegraphics[width=4cm]{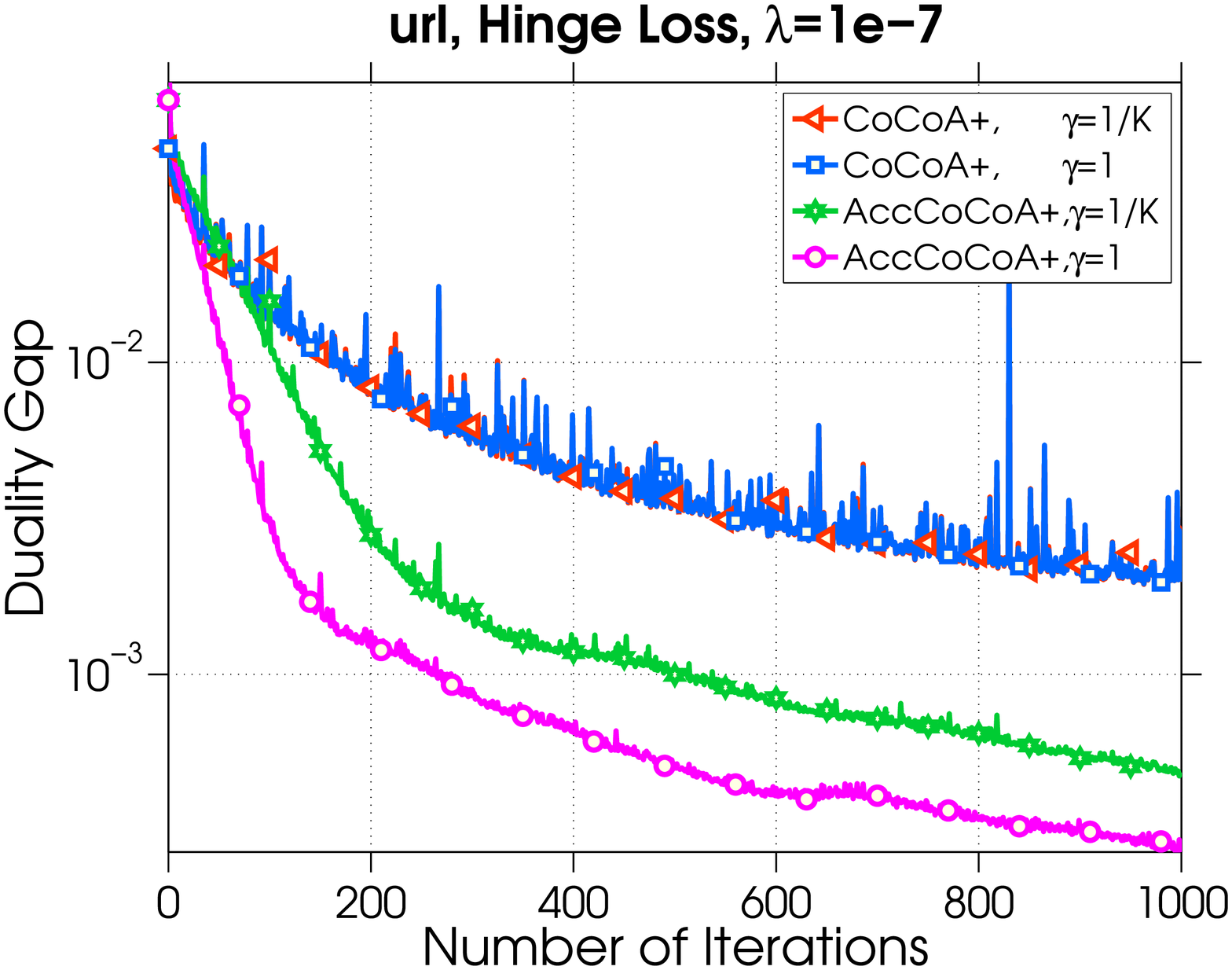}

  \includegraphics[width=4cm]{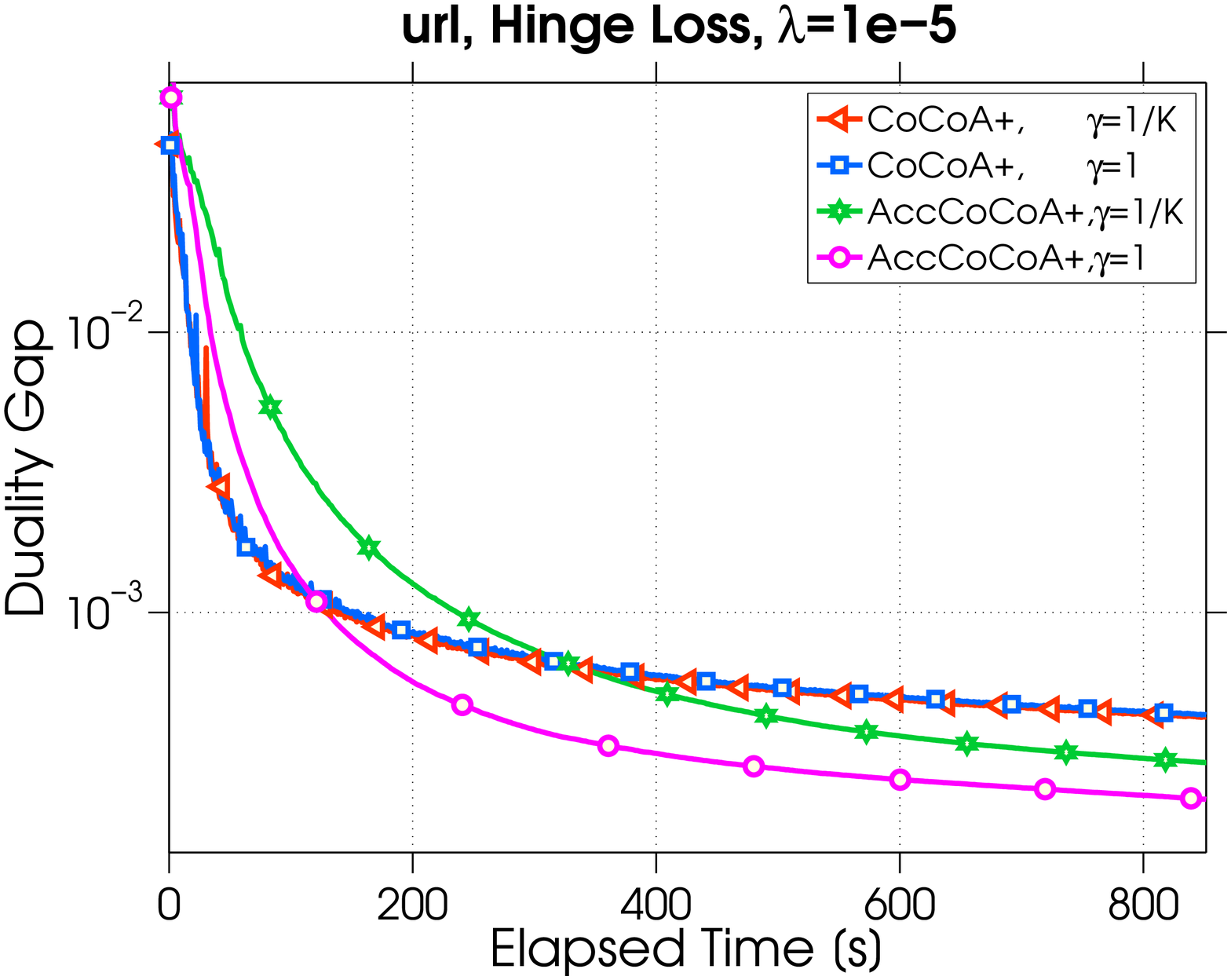}
  \includegraphics[width=4cm]{url_1e-6_2.eps}
  \includegraphics[width=4cm]{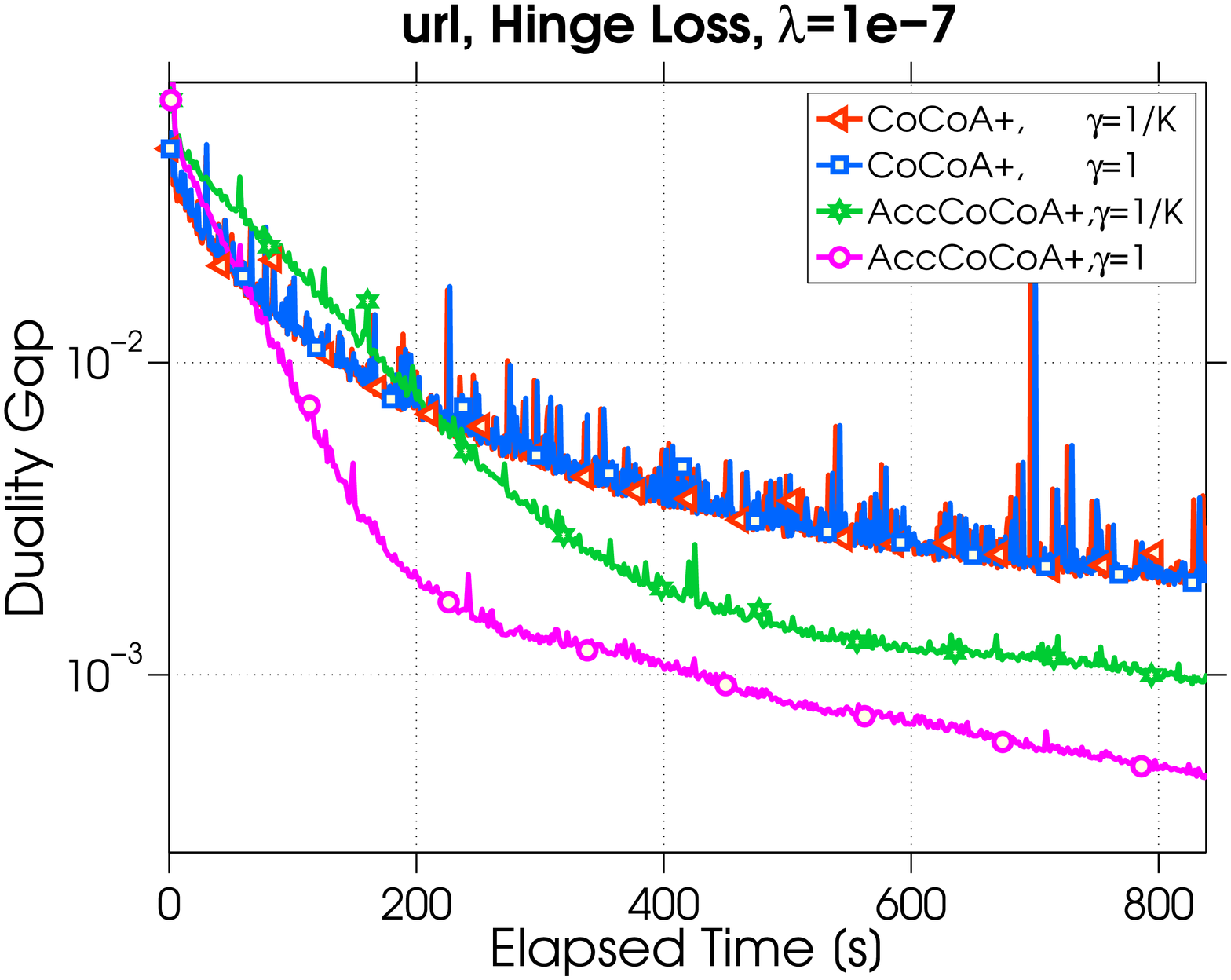}
\caption{Duality gap as a function of iterations (top row) and elapsed time (bottom row) when solving hinge-loss SVM problems with different regularization values ($\lambda$) on the \texttt{url} dataset. }
\label{fig:f2}
\end{figure}

Figure~\ref{fig:lasso1} shows analogous results when the algorithms are employed to solve the Lasso problem in~\eqref{eq:L1}, for which increasing the regularization parameter $\lambda_1$ typically leads to more sparsity of the solution vector.  The results indicate that the accelerated algorithm offers faster convergence of the sub-optimality gap to zero, especially for small values of $\lambda_1$.

\begin{figure}[ht] 
\centering
  \includegraphics[width=4.2cm]{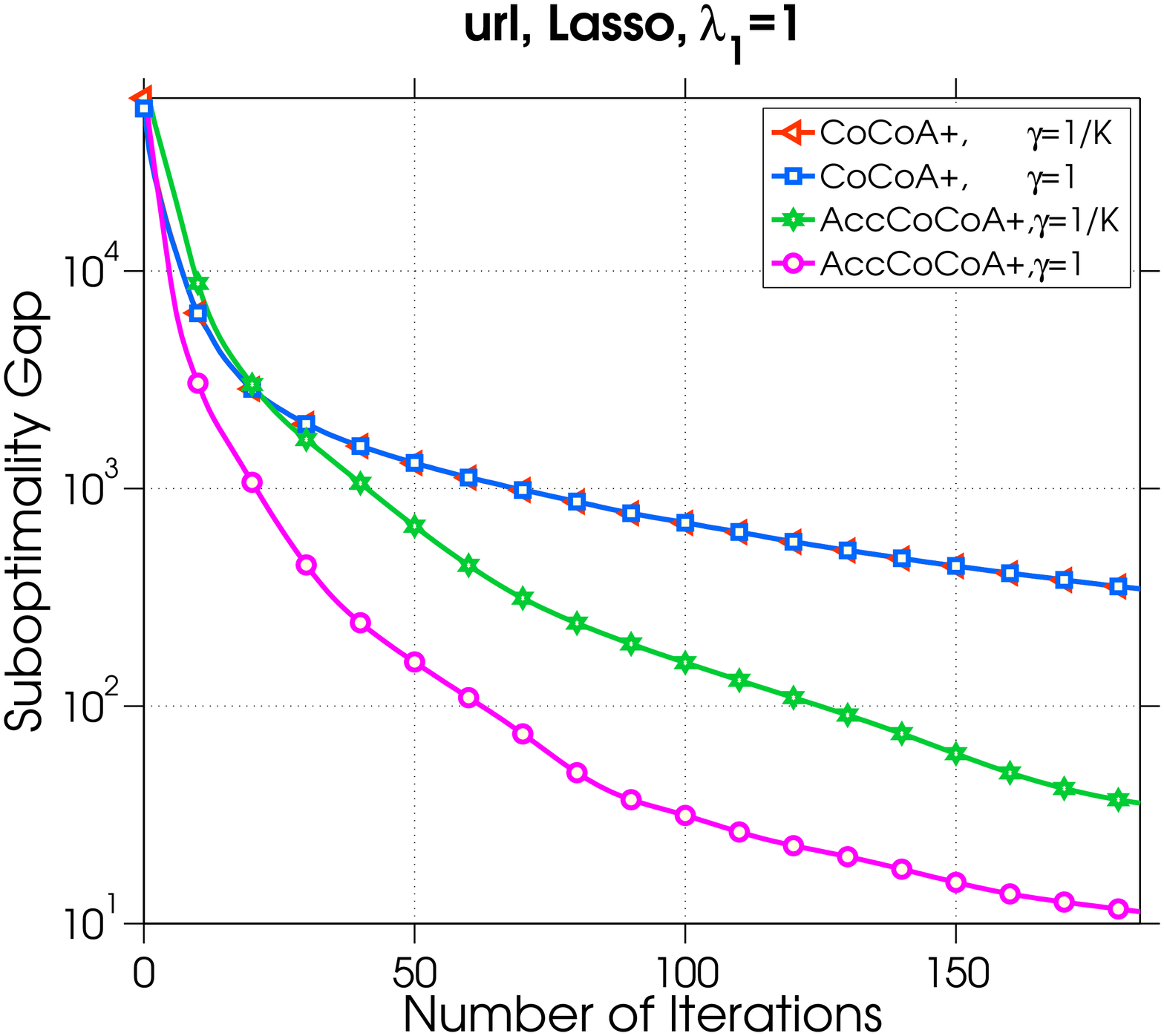}
  \includegraphics[width=4.2cm]{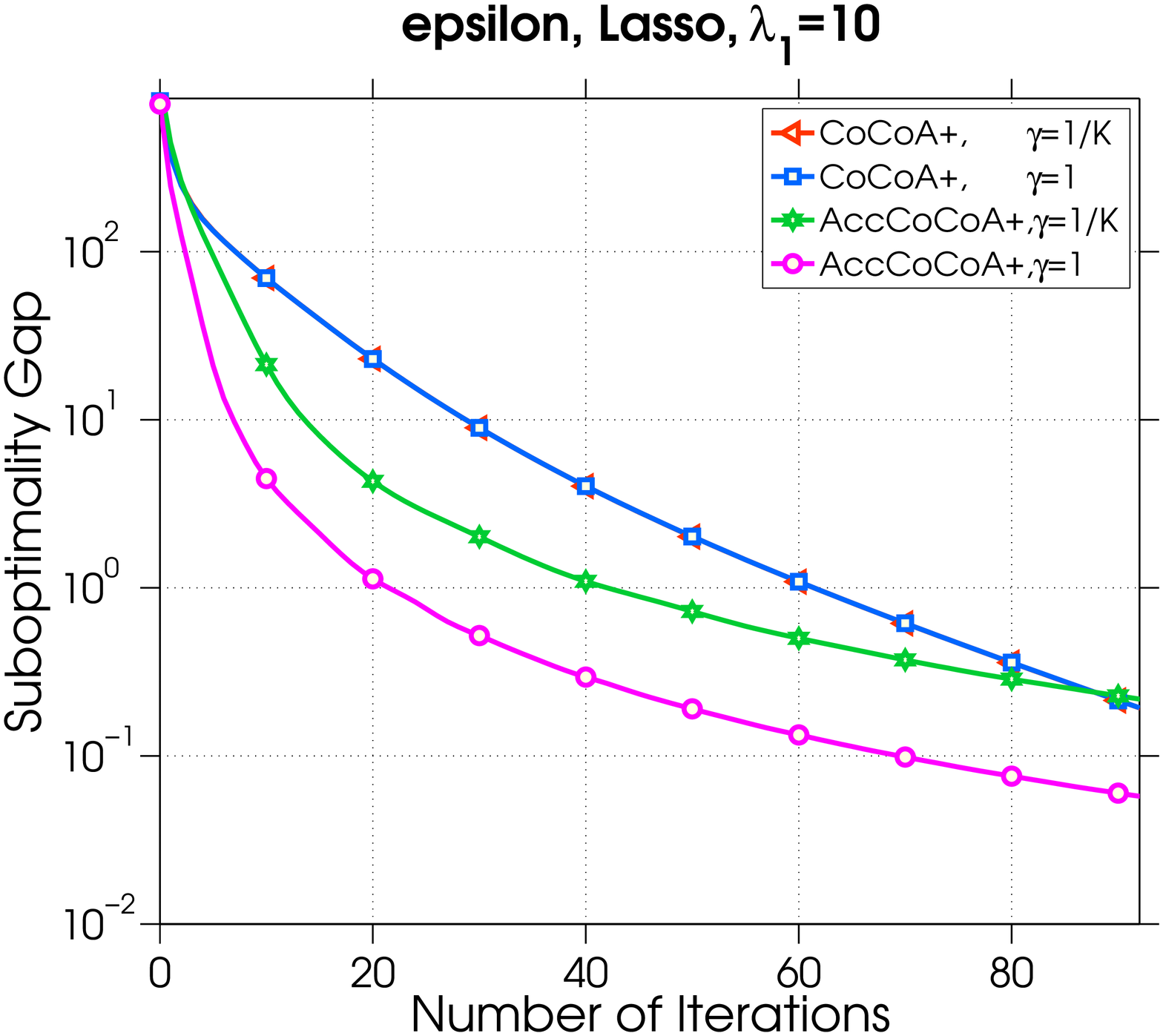}

  \includegraphics[width=4.2cm]{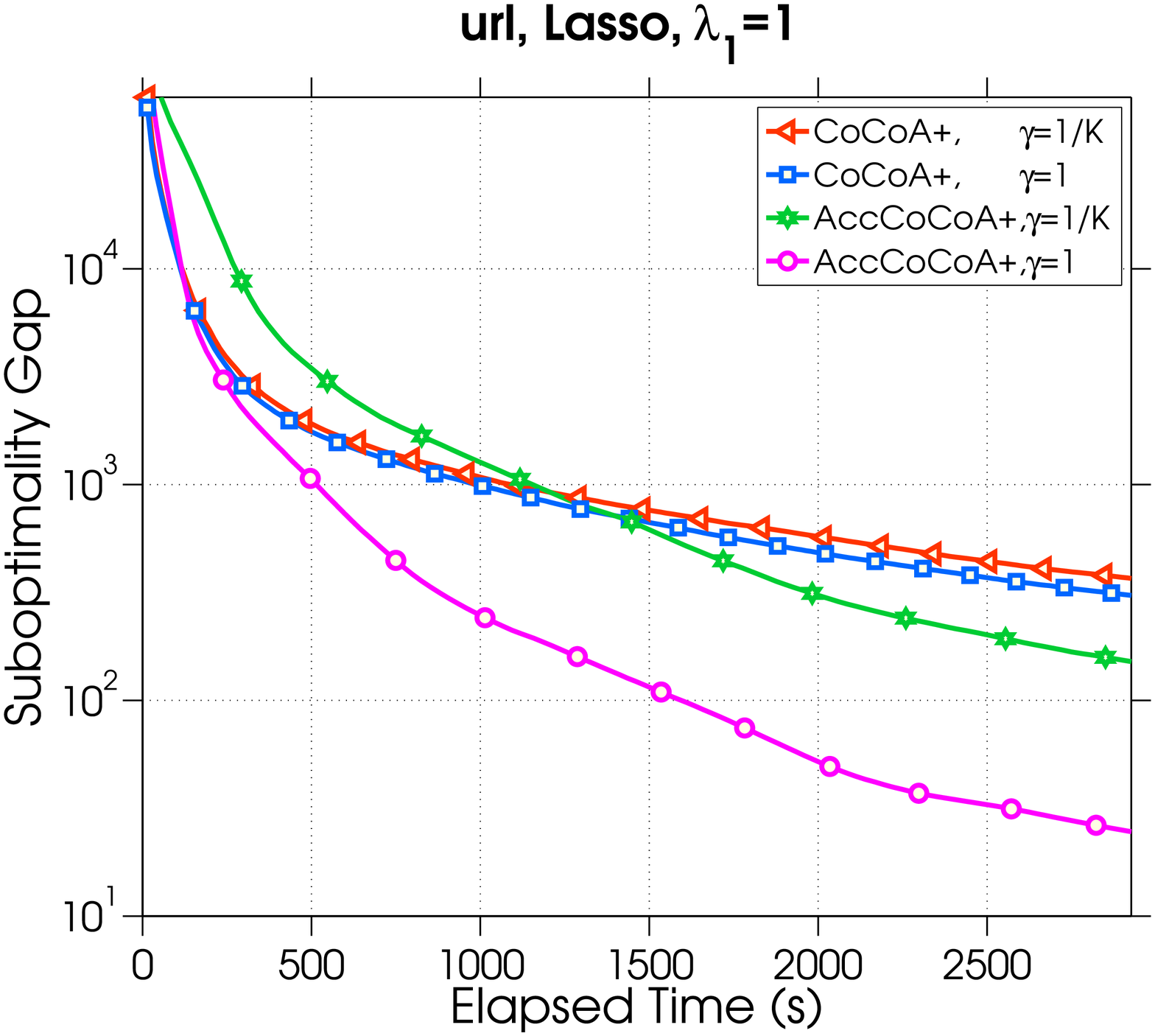}
  \includegraphics[width=4.2cm]{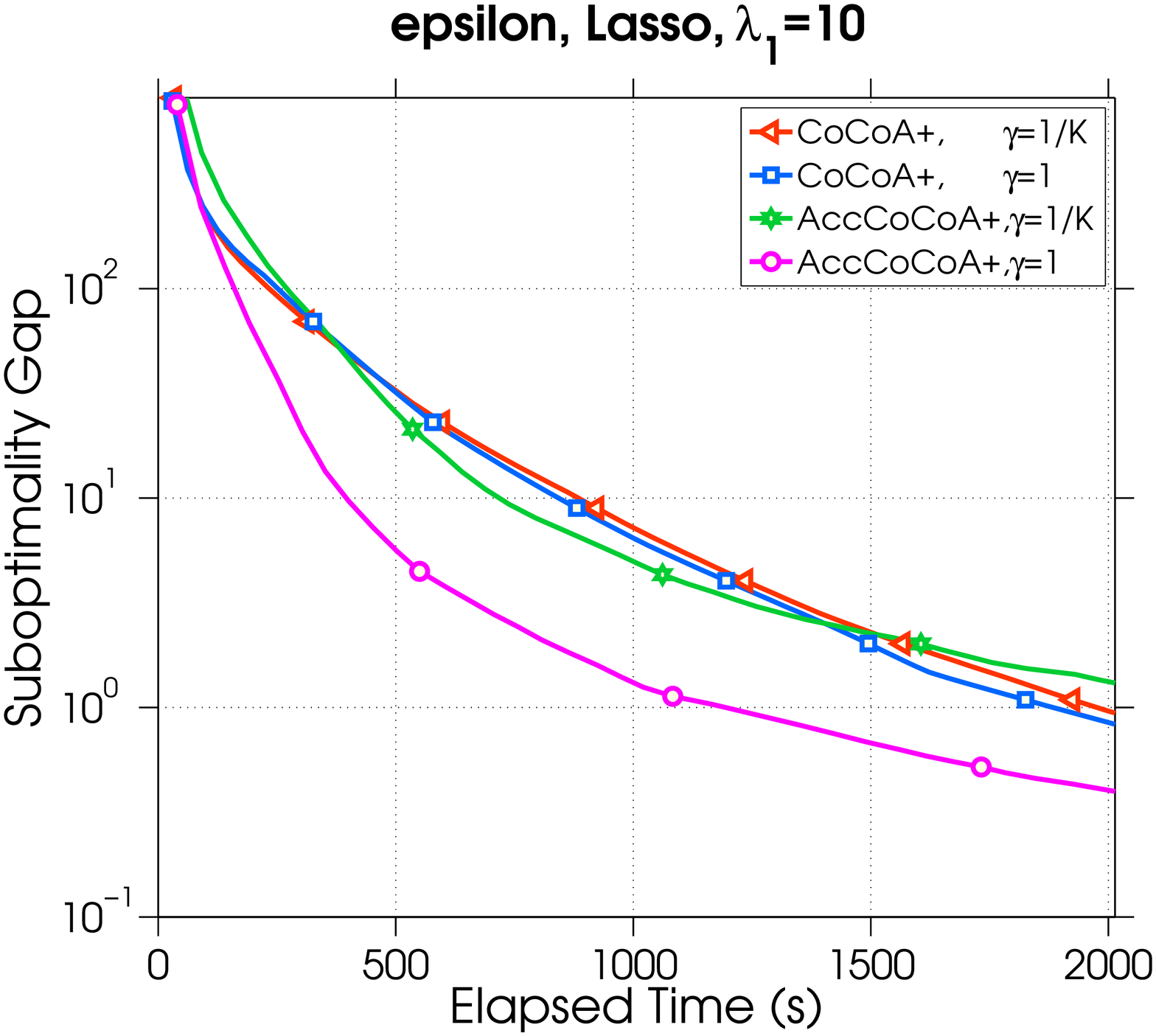}
  \caption{Sub-optimality gap as a function of iterations (top row) and elapsed time (bottom row) when solving Lasso problems.  The choice of the regularizer value $(\lambda_1)$ are such that for the \textit{url} dataset the density of the optimal solution is $5.3\%$, while for the \textit{epsilon} dataset the density is $13.56\%$.}
  \label{fig:lasso1}
\end{figure}

We also ran experiments to demonstrate how the performance of \textsc{AccCoCoA+} depends on the number of iterations ($H$) that SDCA runs for solving each subproblem. Figure~\ref{fig:50} shows that for larger $H$ the subproblems will be solved more accurately, and thus fewer outer iterations can be expected to reach a desired tolerance on the duality gap.  However, in terms of running time, it is not always better to choice larger $H$. For example, for the \texttt{covtype} dataset, choosing $H=10^4$ results in less time to reach a tolerance of $10^{-2}$ than is needed when $H=10^3$ or $H=10^5$.  However, for the $\texttt{news20}$ dataset, the time required decreases with $H$ for all values considered in our experiments.  The reason that this occurs is that $d$ is quite large in this dataset, which makes each round of communication quite time consuming.  Therefore, by solving the subproblems more accurately (by running more iterations), one achieves a better balance between communication and computation.

Such a tradeoff between communication and computation can also be observed in Figure~\ref{fig:50XX}, where we compare how the values of $H$ can affect the convergence of the duality gap for the \texttt{covtype} and \texttt{rcv\_test} datasets. The left two plots illustrate that to reach the same tolerance on the duality gap, the number of iterations can always be reduced by reaching a more accurate solution for each subproblem.  However, the curves in the other two plots indicate that there is always a best value of $H$ that leads to fastest convergence with respect to running time.

\begin{figure}[ht]
\centering
  \includegraphics[width=4.5cm]{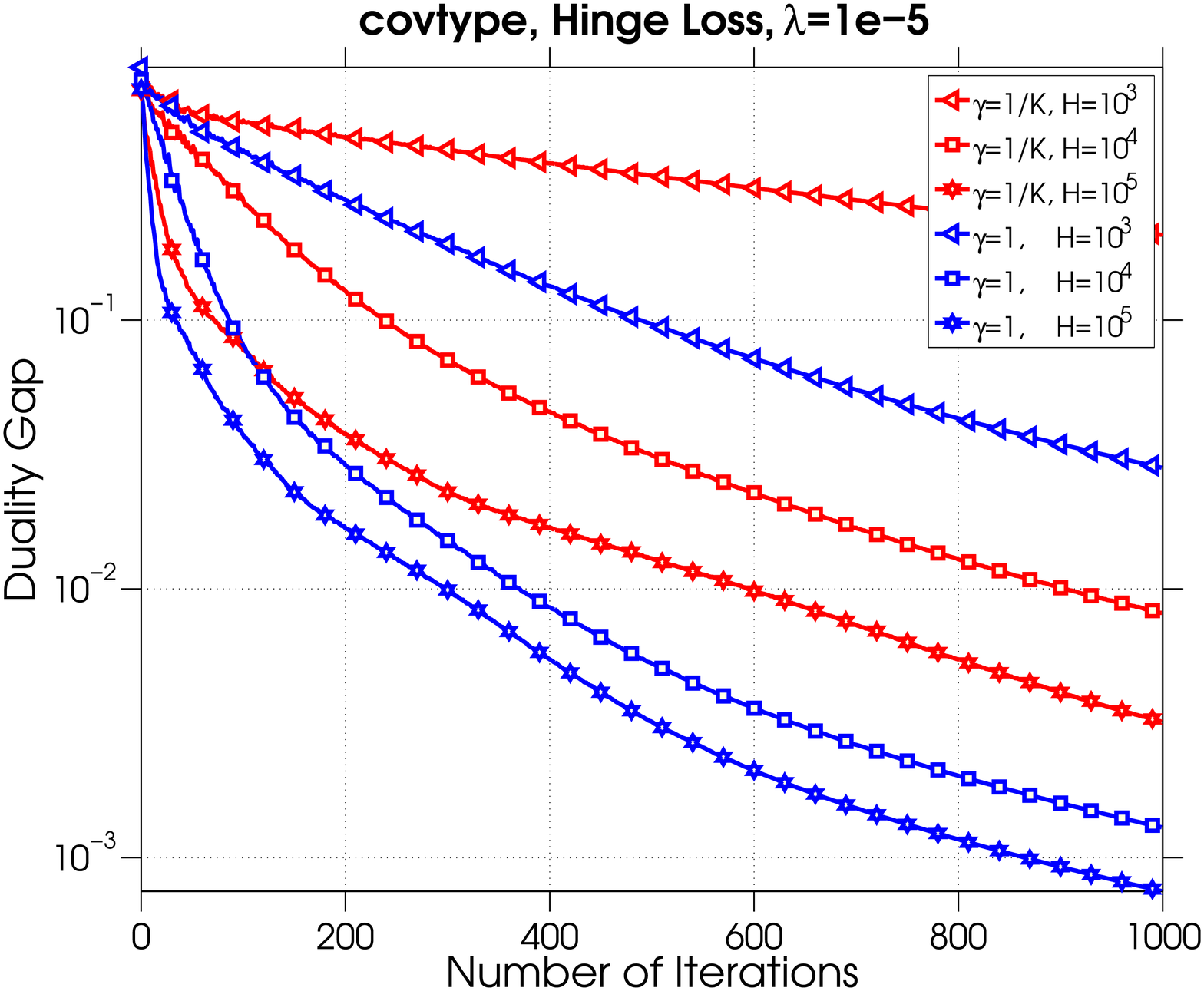}
  \includegraphics[width=4.5cm]{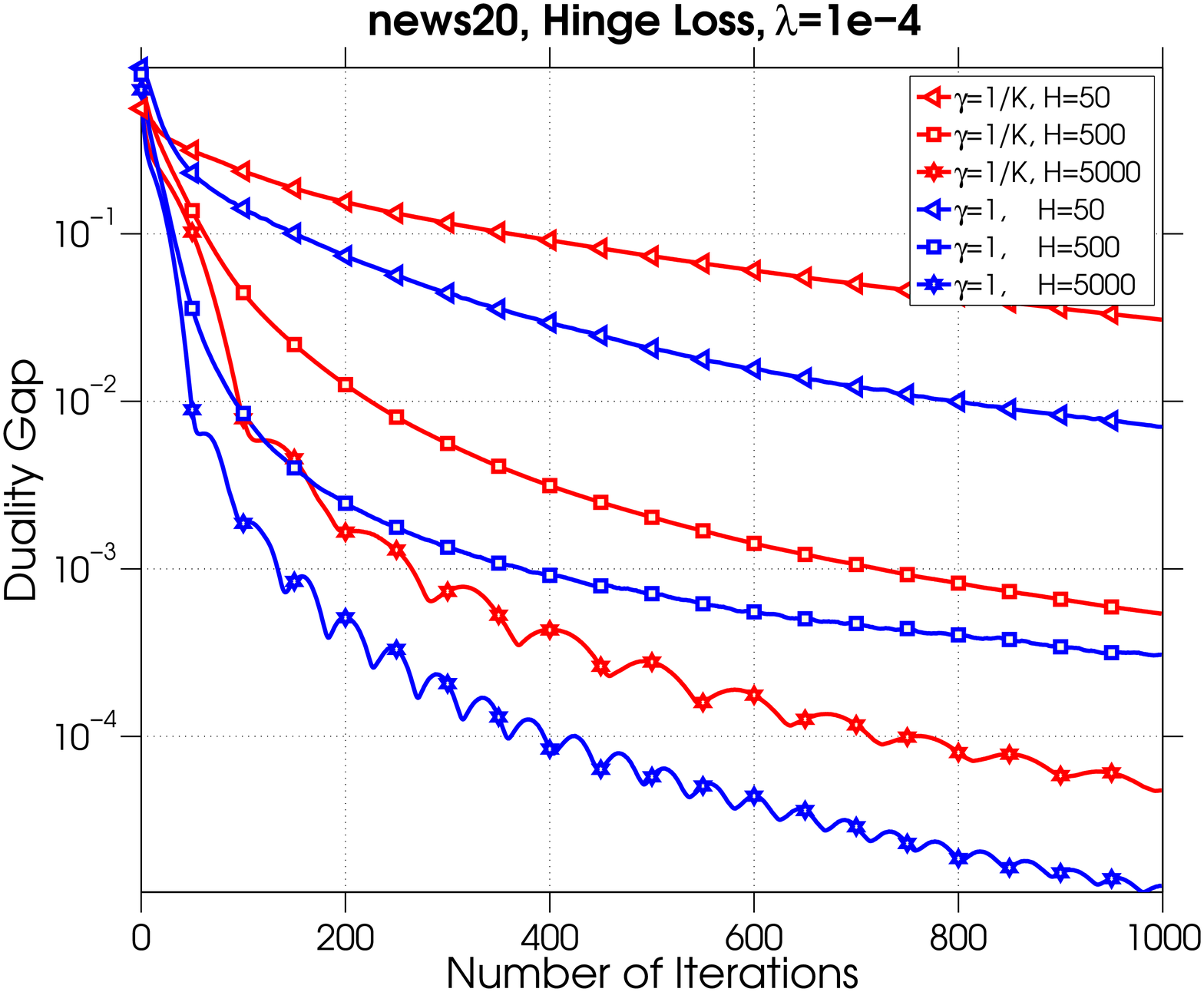}

  \includegraphics[width=4.5cm]{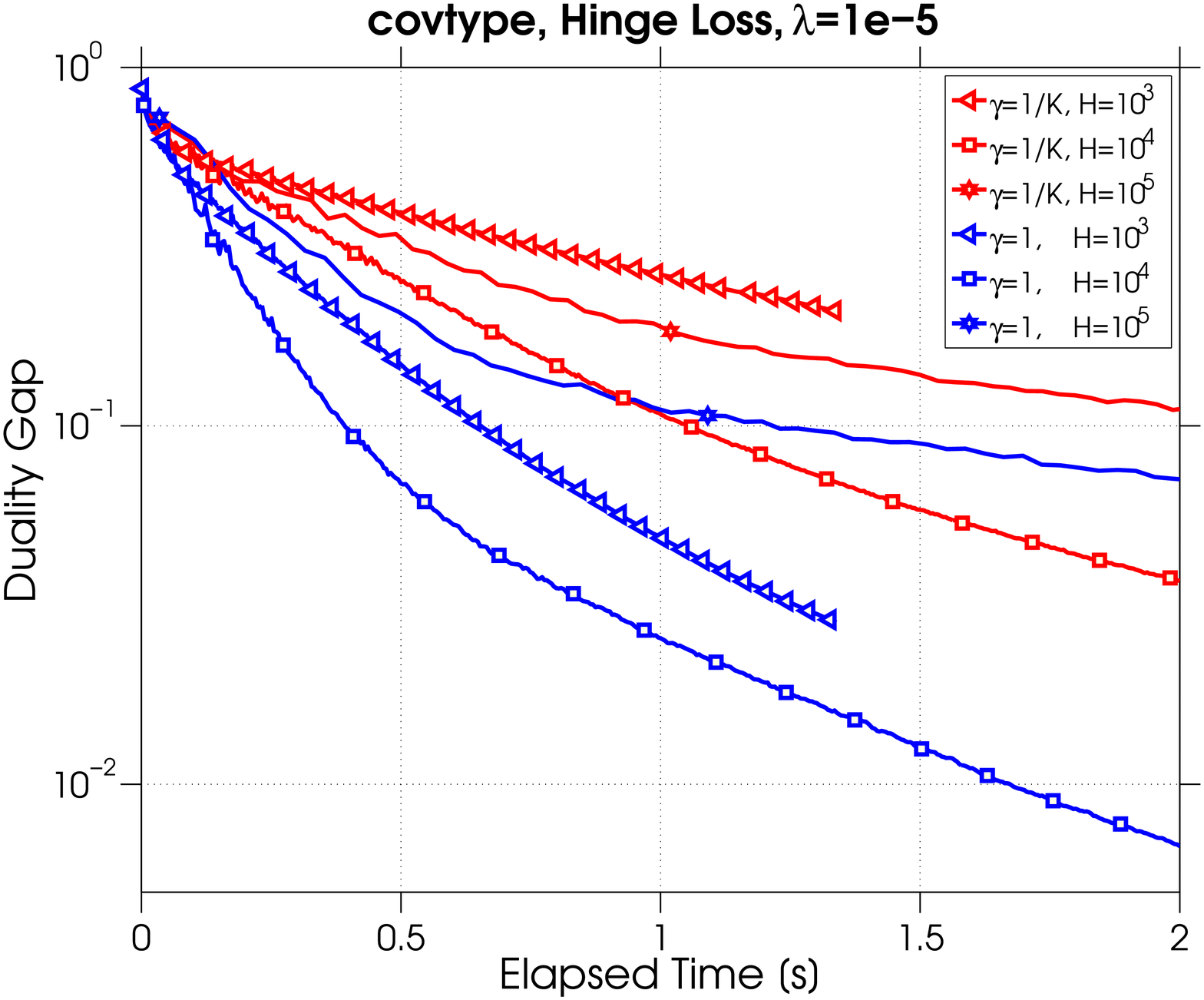}
  \includegraphics[width=4.5cm]{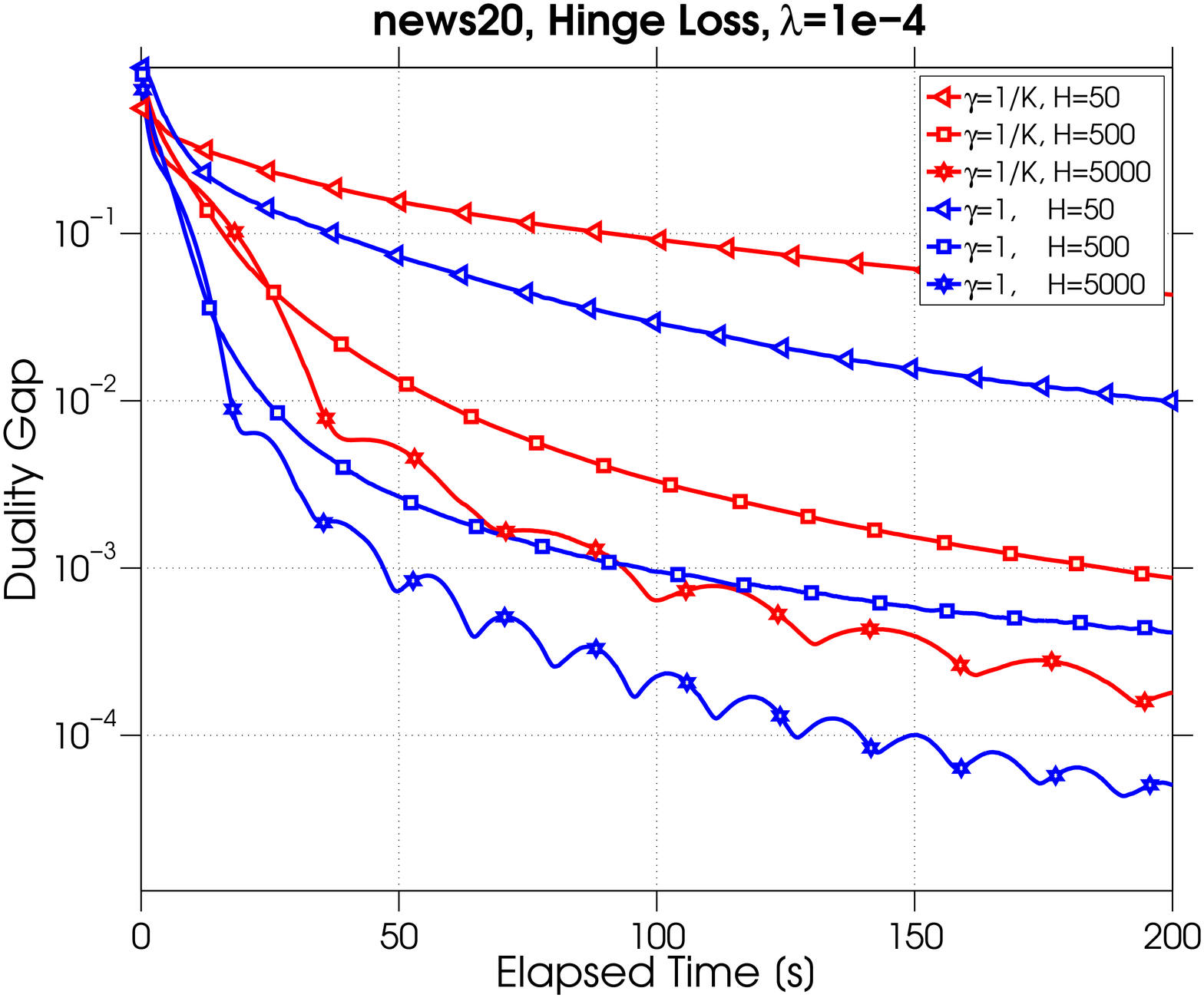}
\caption{Comparison of performance for different inner iteration limits when $\sigma' = \frac{1}{K}$ and $\sigma'=1$.}
$\ $\\
\label{fig:50}
  \includegraphics[width=4.5cm]{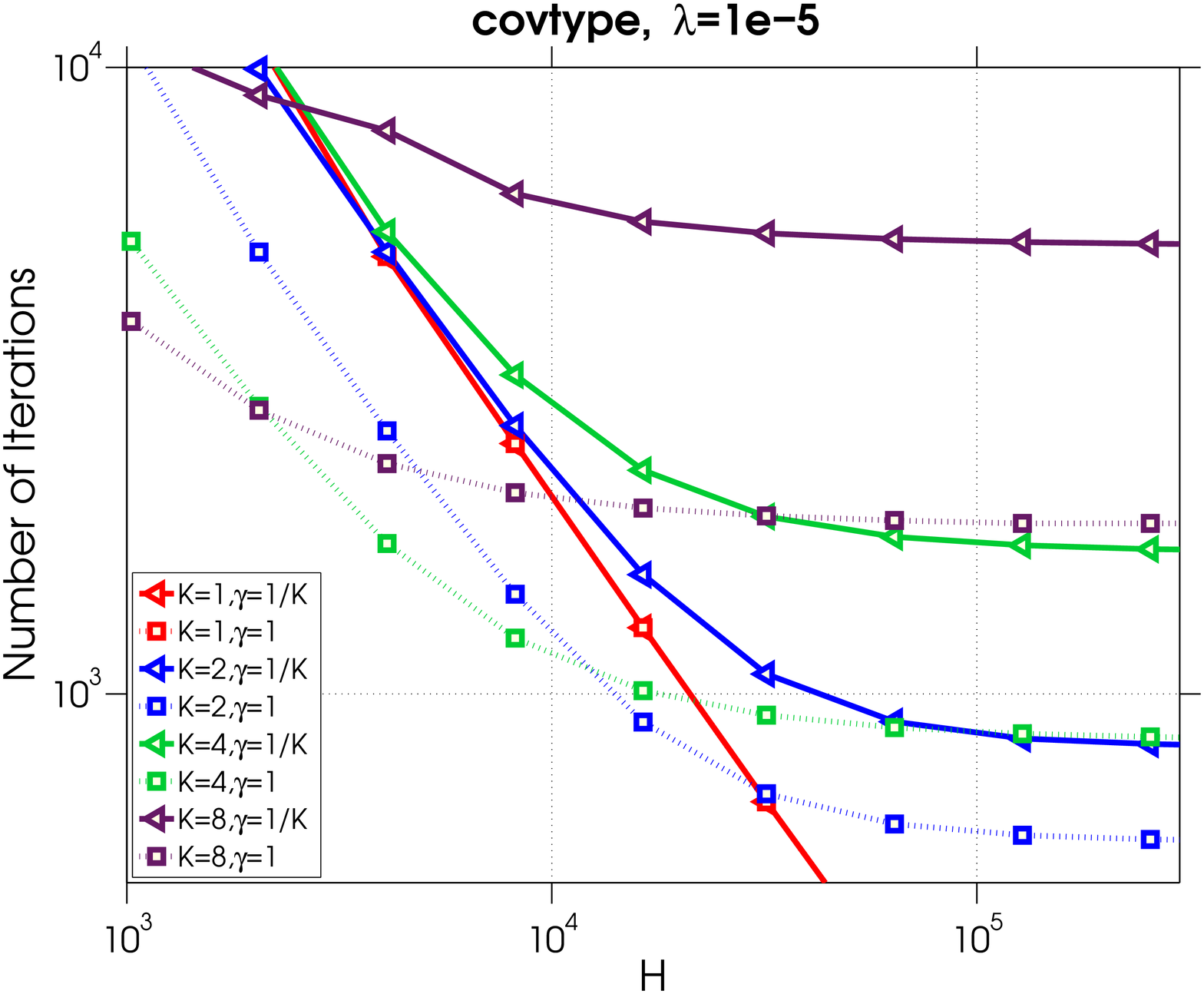}  
  \includegraphics[width=4.5cm]{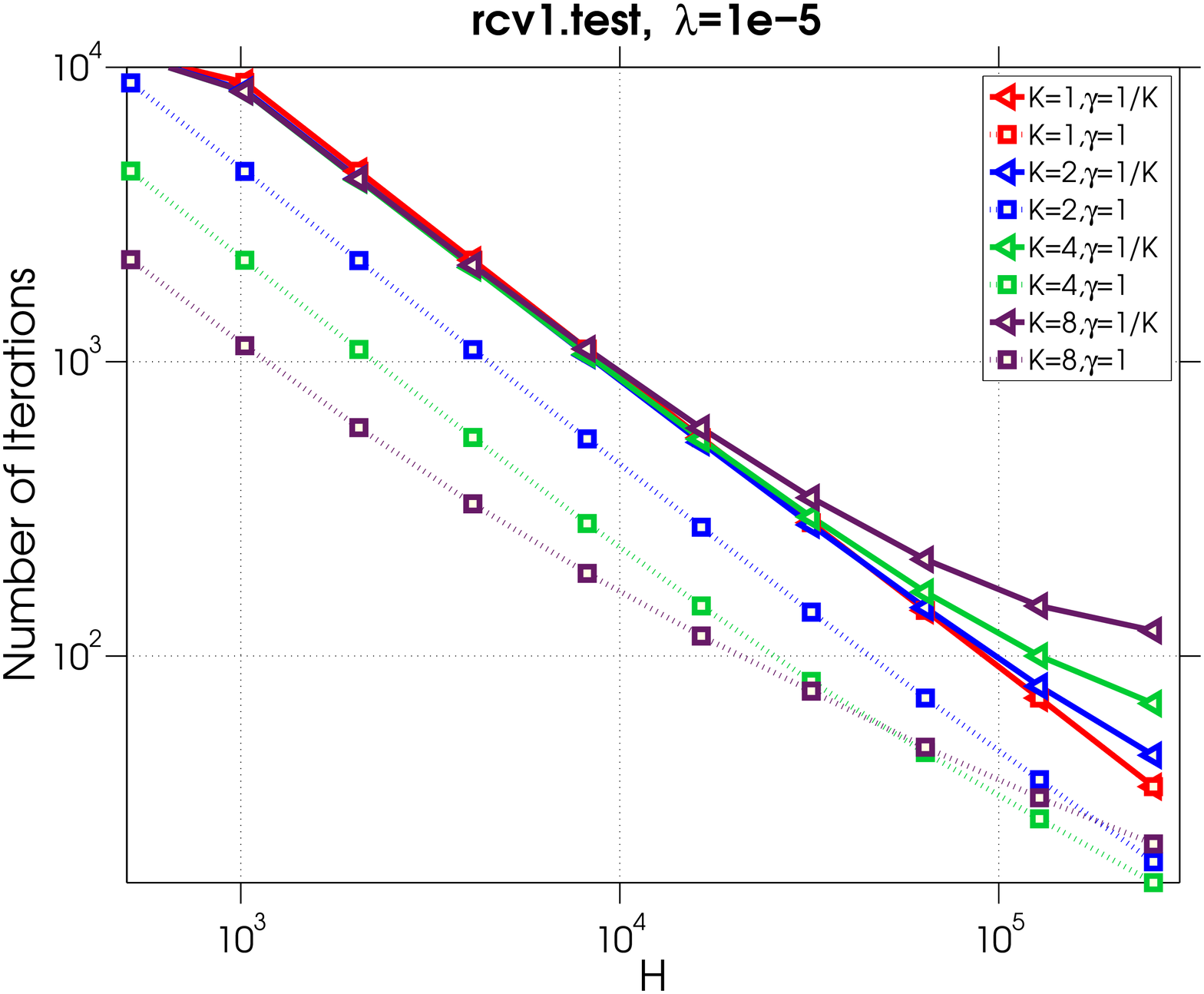}

  \includegraphics[width=4.5cm]{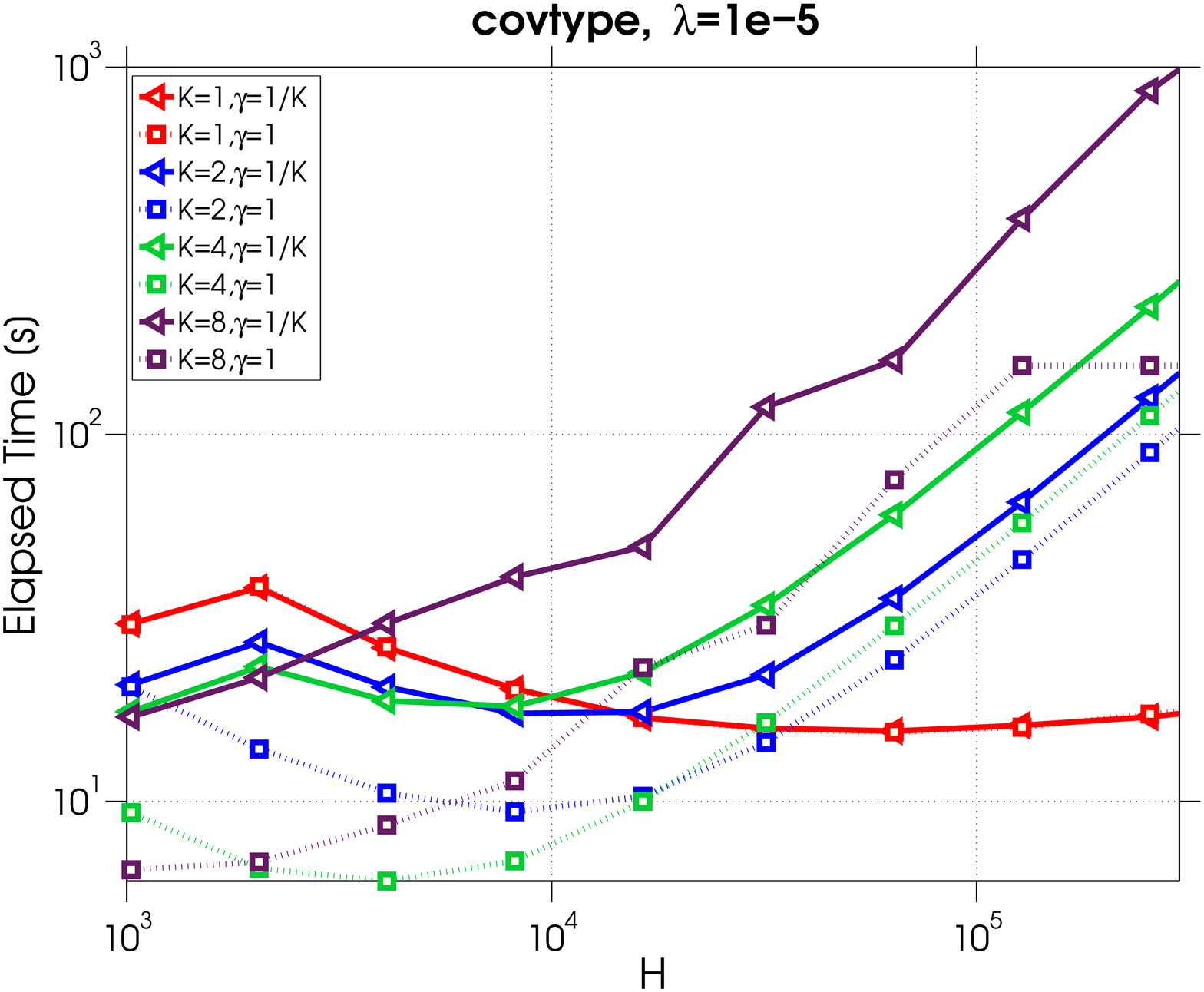}
  \includegraphics[width=4.5cm]{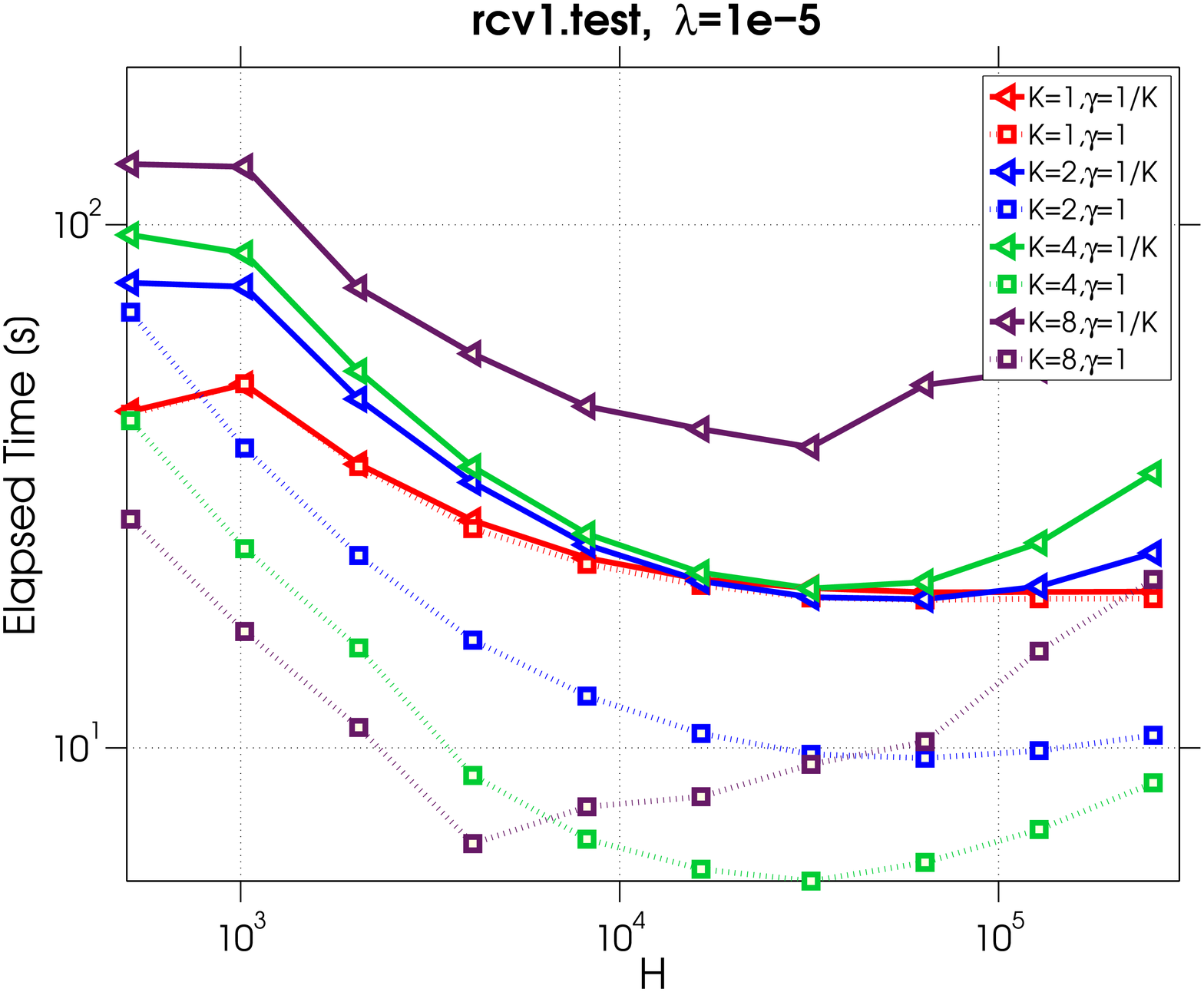} 

\caption{\small Number of iterations and running time required to reach a tolerance of $10^{-3}$ on the duality gap as the inner iteration limit ($H$) is varied.}
\label{fig:50XX}
\end{figure}

Our last experiment shows how \textsc{AccCoCoA+} scales with the number of machines~($K$).  The results are shown in Figure~\ref{fig:f3}. We set $H=\frac{n}{K}$ for every $K$ to make sure that the same amount of data is utilized during each iteration of \textsc{AccCoCoA+} across all $K$ machines. By doing so, it is also expected that each subproblem has been solved to similar accuracy, according to the complexity result of SDCA in \cite{shalev2013stochastic}. The results show that when $\gamma=1$, \textsc{AccCoCoA+} takes almost the same amount of time regardless of how many machines are used, which demonstrates the better scaling properties it has than the $\gamma=\frac1K$ case.

\begin{figure}[ht]
\centering
  \includegraphics[width=4.5cm]{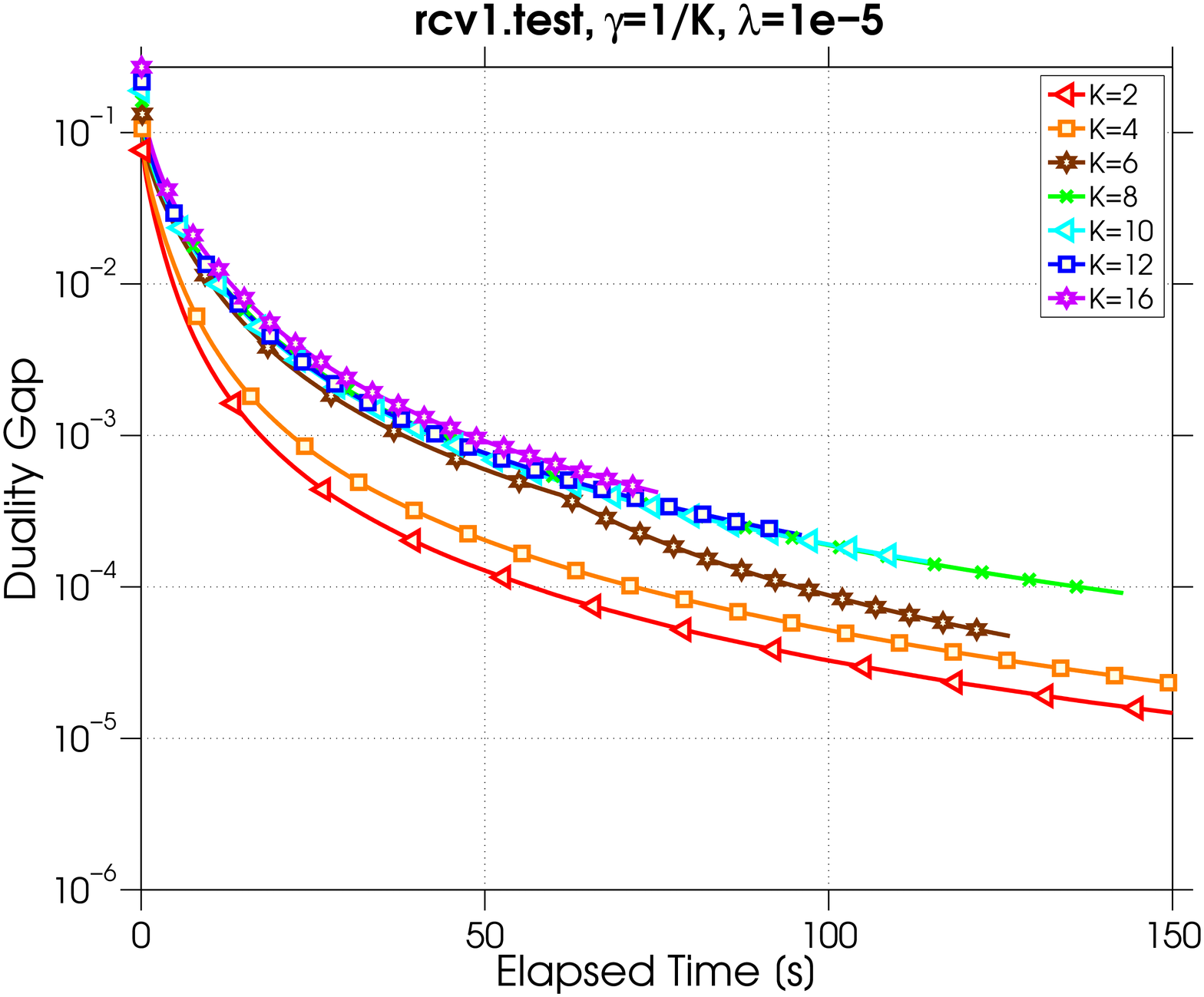}
  \includegraphics[width=4.5cm]{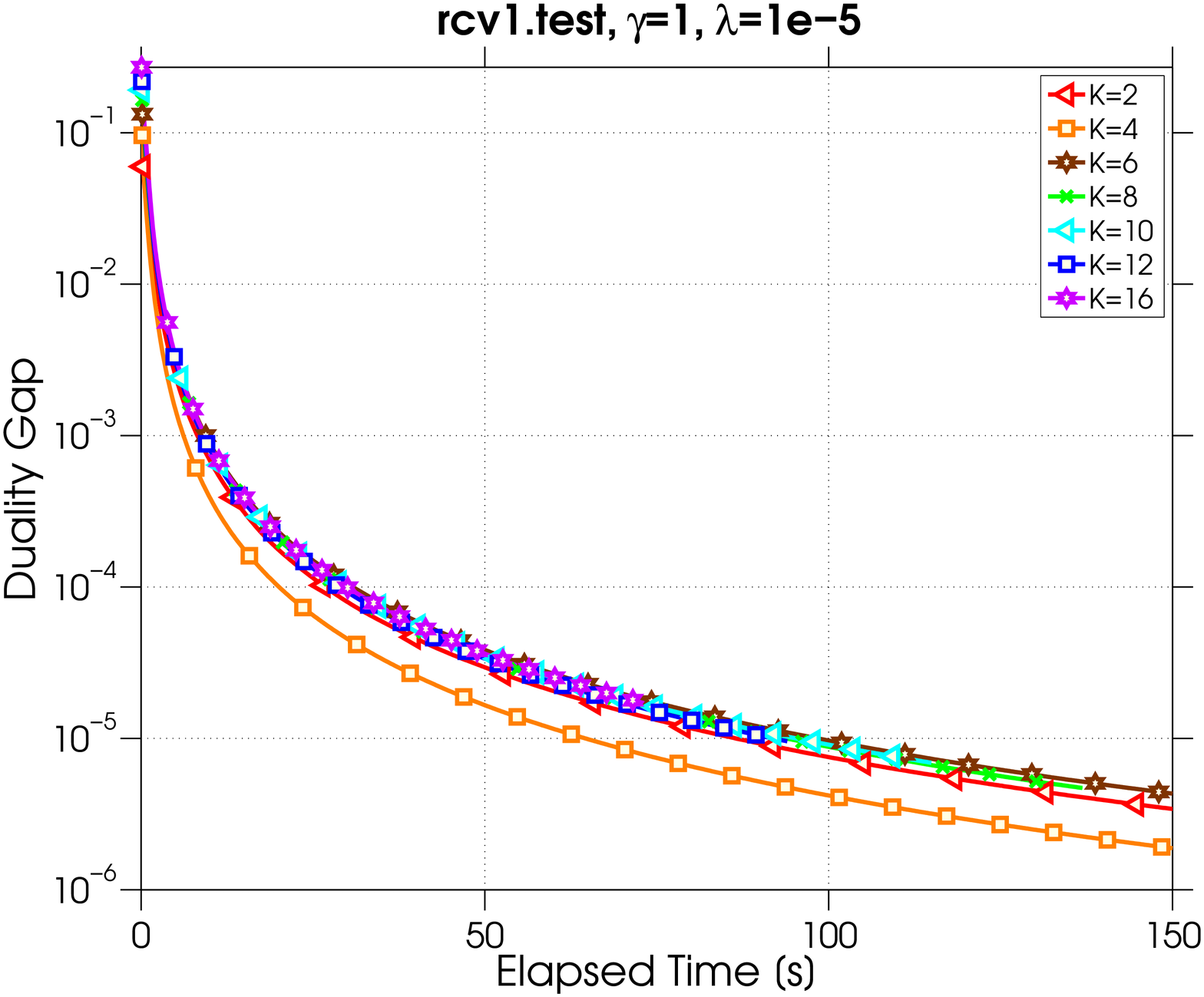}
  \caption{Running times on different numbers of machines $K$.}\label{fig:f3}
\end{figure}

\section{Conclusion} 
We proposed and analyzed \textsc{AccCoCoA+}, an accelerated variant of \textsc{CoCoA}+ achieving the optimal $\mathcal{O}(1/t^2)$ convergence rate. The method is robust to inaccurate subproblems both in theory and practice, and performs well in large-scale experiments. 
  Our analysis provides constants in the convergence rate which are significantly tighter compared to those previously obtained for \textsc{CoCoA+}.

{

\bibliographystyle{plain}
\bibliography{citations.bib}
}

\end{document}